\documentclass[12pt]{article}
\usepackage{latexsym,amsfonts,amsmath,graphics,amsthm}
\usepackage{epsfig}
\usepackage{array}
\usepackage{caption}
\usepackage{graphicx}
\usepackage{epsfig}

\newtheorem{theorem}{Theorem}
\newtheorem{lemma}{Lemma}
\newtheorem{corollary}{Corollary}

\newtheorem{algorithm}{Algorithm}

\newcommand{\satop}[2]{\stackrel{\scriptstyle{#1}}{\scriptstyle{#2}}}

\allowdisplaybreaks

\begin{document}

\title{Random weights, robust lattice rules and the geometry of the cbc$r$c algorithm\footnote{MCS2010: 65D30, 65D32; Keywords: Quasi-Monte Carlo, lattice rule, component-by-component, cbc, weighted space} }

\author{Josef Dick\footnote{School of Mathematics and Statistics, The University of New South Wales, Sydney, NSW 2052, Australia; email josef.dick@unsw.edu.au; Supported by a Queen Elizabeth 2 Fellowship from the Australian Research Council. }}

\date{\today}
\maketitle

\begin{abstract}
In this paper we study lattice rules which are cubature formulae to approximate integrands over the unit cube $[0,1]^s$ from a weighted reproducing kernel Hilbert space. We assume that the weights are independent random variables with a given mean and variance for two reasons stemming from practical applications: (i) It is usually not known in practice how to choose the weights. Thus by assuming that the weights are random variables, we obtain robust  constructions (with respect to the weights) of lattice rules. This, to some extend, removes the necessity to carefully choose the weights. (ii) In practice it is convenient to use the same lattice rule for many different integrands. The best choice of weights for each integrand may vary to some degree, hence considering the weights random variables does justice to how lattice rules are used in applications.

In this paper the worst-case error is therefore a random variable depending on random weights. We show how one can construct lattice rules which perform well for weights taken from a set with large measure. Such lattice rules are therefore robust with respect to certain changes in the weights. The construction algorithm uses the component-by-component (cbc) idea based on two criteria, one using the mean of the worst case error and the second criterion using a bound on the variance of the worst-case error. We call the new algorithm the cbc$2$c (component-by-component with 2 constraints) algorithm.

We also study a generalized version which uses $r$ constraints which we call the cbc$r$c (component-by-component with $r$ constraints) algorithm. We show that lattice rules generated by the cbc$r$c algorithm simultaneously work well for all weights in a subspace spanned by the chosen weights $\boldsymbol{\gamma}^{(1)},\ldots, \boldsymbol{\gamma}^{(r)}$. Thus, in applications, instead of finding one set of weights, it is enough to find an $r$ dimensional convex polytope in which the optimal weights lie. The price for this method is a factor $r$ in the upper bound on the error and in the construction cost of the lattice rule. Thus the burden of determining one set of weights very precisely can be shifted to the construction of good lattice rules.

Numerical results indicate the benefit of using the cbc$2$c algorithm for certain choices of weights.
\end{abstract}

\section{Introduction}

In this paper we study the integration error using a lattice rule. A lattice rules is a quadrature rule of the form
\begin{equation*}
Q_N(f) = \frac{1}{N} \sum_{n=0}^{N-1} f\left(\left\{\frac{n \boldsymbol{g}}{N} \right\}\right),
\end{equation*}
where $N > 1$ and $s$ are natural numbers, $\boldsymbol{g} \in \{1,\ldots, N-1\}^s$, and where $\{\boldsymbol{x}\} = (\{x_1\},\ldots, \{x_s\})$ for $\boldsymbol{x} = (x_1,\ldots, x_s)$ and $\{x\} = x-\lfloor x \rfloor$ stands for the fractional part of a nonnegative real number $x$. Lattice rules are quasi-Monte Carlo algorithms which are useful to approximate integrals $\int_{[0,1]^s} f(\boldsymbol{x}) \,\mathrm{d} \boldsymbol{x}$. In the following we present a survey of the literature and some background on lattice rules.

\subsection{Literature survey and background on lattice rules}

It has been shown that lattice rules are efficient for approximating integrals of periodic functions (see \cite{H98a,K63,N92,SJ94}). The construction of lattice rules has seen many advances in recent years. An important framework in which to study lattice rules (and other quadrature rules) are reproducing kernel Hilbert spaces, which have first been considered in \cite{H98b} and are now a standard tool in quasi-Monte Carlo integration. The component-by-component (cbc) construction (where the generating vector $\boldsymbol{g}$ is constructed one component at a time) was first discovered by Korobov~\cite{K63} and \cite[Theorem~18, p. 120]{K59} and independently rediscovered by Sloan and Reztsov \cite{SR02}. In \cite{SKJ02, SKJ03} this idea has been further developed to allow one to use lattice rules also for nonperiodic integrands. Optimal convergence rates for lattice rules constructed this way have been shown in \cite{K59,K63} and independently in \cite{K03} for a prime number of points $N$ and in \cite{D04} for a nonprime number of points. A breakthrough in reducing the construction cost of the cbc construction has been achieved by Nuyens and Cools in \cite{NC06,NC06b}, who showed how the fast Fourier transform can be used to reduce the construction cost of the search algorithm. A further very important development has been the introduction of weighted function spaces by Sloan and Wo\'zniakowski~\cite{SW98}. Therein, the authors make the important observation that integrands may have different dependence on different projections. To take this fact into account, the authors introduced so-called weighted function spaces which yields a weighted worst-case error criterion. A comprehensive introduction to weighted function spaces and tractability questions as well as further background can be found in the comprehensive monographs \cite{NW08,NW10}.

\subsection{Worst-case error}

In the following we introduce a specific reproducing kernel which will be sufficient to illustrate our algorithm. In order to keep the notation as simple as possible, we do not consider the most general case possible.

Consider the reproducing kernel (see \cite{A50}) $K:[0,1]^2 \to \mathbb{C}$ defined by
\begin{equation*}
K_\gamma(x,y) = \gamma B_2(\{x-y\}),
\end{equation*}
where $\gamma \ge 0$ is a nonnegative real number, the 'weight', and $B_2(z) = z^2-z+1/6$ is the Bernoulli polynomial of degree two (cf. \cite{H98b}). The Bernoulli polynomial $B_2$ has the Fourier series
\begin{equation*}
B_2(w) = \frac{1}{2\pi^2} \sum_{k \in \mathbb{Z} \setminus \{0\}} k^{-2} \mathrm{e}^{2\pi \mathrm{i} k w}.
\end{equation*}
The reproducing kernel $K_\gamma$ defines a reproducing kernel Hilbert space $\mathcal{H}_\gamma$ of absolutely continuous, periodic functions on $[0,1]$ which integrate to $0$, with inner product
\begin{equation*}
\langle f, g \rangle = \frac{2\pi^2}{\gamma} \sum_{k \in \mathbb{Z} \setminus \{0\}} k^2 \, \widehat{f}(k) \overline{\widehat{g}(k)}.
\end{equation*}
For dimensions $s > 1$ we consider the reproducing kernel
\begin{equation*}
K_{\boldsymbol{\gamma}}(\boldsymbol{x},\boldsymbol{y}) = 1+ \sum_{\emptyset \neq u \subseteq \mathcal{S}} \gamma_u \prod_{i\in u} B_2(\{x_i-y_i\}),
\end{equation*}
where $\mathcal{S} = \{1,\ldots, s\}$, $\boldsymbol{\gamma} = (\gamma_u)_{\emptyset \neq u \subseteq \mathcal{S}}$ is a set of nonnegative real numbers $\gamma_u$ associated with the projection onto the coordinates in $u$ (we refer to these numbers as the 'weights' \cite{SW98}), $\boldsymbol{x} = (x_1,\ldots, x_s)$ and $\boldsymbol{y}=(y_1,\ldots, y_s)$. The associated reproducing kernel Hilbert space is denoted by $\mathcal{H}_{\boldsymbol{\gamma}}$ which is a sum of tensor products of the reproducing kernel Hilbert space with kernel $K_\gamma$ and the space of constant functions, see \cite{H98a, H98b, SW98} for more information. The kernel $K_{\boldsymbol{\gamma}}$ can also be interpreted as the shift-invariant kernel of a reproducing kernel Hilbert space of non-periodic functions \cite{HW00}. Thus the results here can also be interpreted for randomly shifted lattice rules in the associated non-periodic reproducing kernel Hilbert space (as for instance in\cite{K03,NC06,NC06b} and many other papers).

The integration error using a lattice rule with generating vector $\boldsymbol{g} \in \{1,\ldots, N-1\}^s$ is defined as
\begin{equation*}
e(\mathcal{H}_{\boldsymbol{\gamma}}, P_N(\boldsymbol{g})) = \sup_{\satop{f \in \mathcal{H}_{\boldsymbol{\gamma}}, P_N(\boldsymbol{g})}{\|f\|_{\mathcal{H}_{\boldsymbol{\gamma}}} \le 1}} \left|\int_{[0,1]^s} f(\boldsymbol{x})\,\mathrm{d} \boldsymbol{x} - \frac{1}{N} \sum_{n=0}^{N-1} f\left(\left\{\frac{n \boldsymbol{g}}{N} \right\}\right)\right|.
\end{equation*}
It was shown in \cite{H98a,H98b} that
\begin{align*}
e^2(\mathcal{H}_{\boldsymbol{\gamma}}, P_N(\boldsymbol{g})) & = \int_{[0,1]^s} \int_{[0,1]^s} K_{\boldsymbol{\gamma}}(\boldsymbol{x}, \boldsymbol{y}) \,\mathrm{d} \boldsymbol{x} \,\mathrm{d} \boldsymbol{y} \\ & - \frac{2}{N} \sum_{n=0}^{N-1} \int_{[0,1]^s} K_{\boldsymbol{\gamma}}(\boldsymbol{x}, \{n \boldsymbol{x}/N\}) \,\mathrm{d} \boldsymbol{x} + \frac{1}{N^2} \sum_{n,n'=0}^{N-1} K_{\boldsymbol{\gamma}}(\{n \boldsymbol{g}/N\}, \{n' \boldsymbol{g}/N\}) \\ & = \sum_{\emptyset \neq u \subseteq \mathcal{S}} \gamma_u \frac{1}{N} \sum_{n=0}^{N-1} \prod_{i\in u} B_2(\{ n g_i/N\}),
\end{align*}
where the last inequality follows from $\int_0^1 K_\gamma(x,y) \,\mathrm{d} x = 1$ and the fact that (see \cite{H98a,H98b}) $$\frac{1}{N^2} \sum_{n,n'=0}^{N-1} K_{\boldsymbol{\gamma}}(\{n \boldsymbol{g}/N\}, \{n' \boldsymbol{g}/N\}) = \frac{1}{N} \sum_{n=0}^{N-1} K_{\boldsymbol{\gamma}}(\{n \boldsymbol{g}/N\}, \boldsymbol{0}).$$

We use the last expression as error criterion in this paper, i.e.
\begin{equation}\label{eq_wce}
e^2(\mathcal{H}_{\boldsymbol{\gamma}},P_N(\boldsymbol{g})) = \sum_{\emptyset \neq u \subseteq\mathcal{S}} \left(\gamma_u \frac{1}{N} \sum_{n=0}^{N-1} \prod_{i \in u} B_2(\{n g_i/N\}) \right).
\end{equation}
As indicated above, \eqref{eq_wce} can be interpreted as the square worst-case error in a Korobov space or the mean square worst-case error of a randomly shifted lattice rules in a Sobolev space. The square worst-case error $e^2(\mathcal{H},P_N(\boldsymbol{g}))$ is commonly used as error-criterion in a cbc construction \cite{D04,K03,NC06,NC06b,SKJ02,SKJ03}.

The cbc algorithm is a greedy search algorithm to find a good generating vector $\boldsymbol{g}^\ast = (g_1^\ast,\ldots, g_s^\ast) \in \{1,\ldots, N-1\}$. This algorithm works the following way. First one chooses a number of points $N$, the dimension $s$ and some weights $\boldsymbol{\gamma}$. The component-by-component construction then finds a generating vector $\boldsymbol{g}^\ast = (g_1^\ast, \ldots, g_s^\ast)$ in the following way:
\begin{itemize}
\item Set $g_1^\ast = 1$.
\item For $j=1,\ldots, s$ set $$g_j^\ast = \mathrm{argmin}_{1 \le z \le N-1} e^2(\mathcal{H}_{\boldsymbol{\gamma}}, P_N((g_1^\ast,\ldots, g_{j-1}^\ast, z))).$$
\end{itemize}

\subsection{The aim of the paper}

In practice, it is usually not known how to choose the weights $\gamma_u$, $\emptyset \neq u \subseteq \mathcal{S}$ in the worst-case error criterion. Some suggestions on how to choose the weights in financial applications have been put forward for instance in \cite{W07, WS06}. Another method of choosing the weights is by choosing them such that the error bound is minimized \cite{LLS03}. However, choosing good weights remains a particular challenge for the application of lattice rules.

In this paper we assume that the weights are independent random variables with a given mean and variance for two reasons stemming from practical applications:
\begin{itemize}
\item It is usually not known in practice how to choose the weights precisely. By assuming randomness in the weights permits a 'measurement error' or noise in choosing the weights.
\item It is convenient to use the same lattice rule for many different integrands. The best choice of weights for each integrand may vary to some degree, hence considering the weights random variables seems to be the right model in this case.
\end{itemize}

Indeed, it is desirable to have quadrature rules (lattice rules) which are robust with respect to the weights $(\gamma_u)$, that is, for which one obtains a good convergence behavior, not only for one given choice of weights, but for a whole range of weights. In order to construct lattice rules which have this property, we assume that the weights are not given (or fixed), but rather, we assume they are chosen randomly with a given mean and variance. In this way, the square worst-case error $e^2(\mathcal{H}_{\boldsymbol{\gamma}}, P_N(\boldsymbol{g}))$ is a random variable (with respect to the weights $\boldsymbol{\gamma}$). In the following we propose an algorithm to construct lattice rules which have a small expectation value and, at the same time, a small variance of $e^2(\mathcal{H}_{\boldsymbol{\gamma}}, P_N(\boldsymbol{g}))$ with respect to the random choices $\boldsymbol{\gamma}$. In a nutshell, the existence of such a lattice rule is guaranteed by the fact that more than half of the generating vectors have small expectation value of the worst-case error and more than half of the generating vectors have small variance of the worst-case error. Thus there exists at least one vector for which both, the expectation value and the variance, are small. This yields a component-by-component algorithm with $2$ constrains, which we call the cbc$2$c algorithm.

We also study a general version which uses $r$ constraints which we call the cbc$r$c (component-by-component (with) $r$ constraints) algorithm. We show that lattice rules generated by the cbc$r$c algorithm simultaneously work well for all weights in a subspace spanned by the chosen weights $\boldsymbol{\gamma}^{(1)},\ldots, \boldsymbol{\gamma}^{(r)}$. Thus, in applications, instead of finding one set of weights, it is enough to find an $r$ dimensional convex polytope in which the optimal weights lie. The price for this method is a factor $r$ in the upper bound on the error and the construction cost of the lattice rule. Thus the burden of finding one set of weights can be shifted to the construction of good lattice rules.

Theoretically one could make an exhaustive search to obtain a lattice rule which simultaneously works well for all choices of weights. This may eventually shift the question of how to choose the weights for a particular problem to the problem of finding a universal lattice rule which simultaneously works well for all choices of weights, thereby removing the need to choose weights in the first place. The computational challenge though is, that for higher dimensions finding such a lattice rule is currently intractable (since the cost depends exponentially on the dimension). Further, also the upper bound from this paper depends exponentially on the dimension when $r = 2^s-1$. This method may be useful though for integrands with low truncation dimension $d$ by choosing $r = 2^d-1$ in this case.

We note that a similar theory can be applied to polynomial lattice rules and related point sets \cite{DP10,N92}.

In the next section we study the implications of the assumption that the weights are random on the square worst-case error. In Section~\ref{sec3} we first repeat some important insights from \cite{NC06,NC06b}. We introduce the cbc$2$ algorithm and show that the constructed lattice rules work well for weights taken from a set of large measure. We then consider the cbc$r$ algorithm and consider the geometrical interpretation of the algorithm. It is shown that the square worst case error satisfies a certain bound for all weights in an $r$-dimensional convex polytope which is defined by the weights used in the cbc$r$c algorithm. In particular we explain how the weights in the cbc$r$c algorithm determine the shape of the search space of the generating vectors. In Section~\ref{sec_num} we provide some numerical examples to illustrate that in certain instances the cbc$2$c algorithm is beneficial.

\section{The expectation value and standard deviation of the square worst-case error with random weights}

We assume that the weights $\gamma_u$, $\emptyset \neq u \subseteq \mathcal{S}$, are nonnegative, independent random variables with a given mean and variance. Let $\mathbb{E}$ denote the expectation value and $\mathrm{Var}$ the variance. For any $\emptyset \neq u \subseteq \mathcal{S}$ there are numbers $\overline{\gamma}_u, \sigma_u \ge 0$ such that for all $\emptyset \neq u, u' \subseteq \mathcal{S}$ the following properties hold:
\begin{itemize}
\item $\gamma_u \ge 0$,
\item $\mathbb{E}(\gamma_u) = \overline{\gamma}_u$,
\item $\mathbb{E}(\gamma_u \gamma_{u'}) = \mathbb{E}(\gamma_u) \mathbb{E}(\gamma_{u'})$ for $u \neq u'$,
\item $\mathrm{Var}(\gamma_u) = \sigma_u^2$;
\end{itemize}
Note that we have $\mathbb{E}(\gamma_u^2) = \overline{\gamma}_u^2 + \sigma_u^2$. Let $\overline{\boldsymbol{\gamma}} = (\overline{\gamma}_u)_{\emptyset \neq u \subseteq \mathcal{S}}$ and $\boldsymbol{\sigma}  = (\sigma_u)_{\emptyset \neq u \subseteq \mathcal{S}}$.

We point out that the conditions for product weights need to be dealt with carefully. Assume that $\gamma_u = \prod_{i\in u} z_i$ for some nonnegative real numbers $z_i$. Further assume that $\mathbb{E}(z_i) = \overline{z}_i$. Then $\mathbb{E}(\gamma_u) = \prod_{i\in u} \overline{z}_i$ and hence $\overline{\gamma}_u = \prod_{i\in u} \overline{z}_i$. Assume that $\mathrm{Var}(z_i) = y_i^2$. Then for $\emptyset \neq u, u' \subseteq \mathcal{S}$ with $v = u \cap u'$ we have
\begin{equation*}
\mathbb{E}(\gamma_u \gamma_{u'}) = \mathbb{E}(\prod_{i \in u} z_i \prod_{i' \in u'} z_i) = \prod_{i \in u \setminus v} \overline{z}_i \prod_{i' \in u' \setminus v} \overline{z}_{i'} \prod_{i'' \in v} y_i.
\end{equation*}
Thus, using these assumption, we do not have $\mathbb{E}(\gamma_u \gamma_{u'}) = \mathbb{E}(\gamma_u) \mathbb{E}(\gamma_{u'})$ in general, i.e., the weights are not independent (as is obvious from the definition $\gamma_u = \prod_{i\in u} z_i$). The analysis for such weights is slightly different and is not considered here. Instead, if the weights are of product form $\gamma_u = \prod_{i\in u} z_i$, we assume that $\overline{\gamma}_u = \prod_{i\in u} \overline{z}_i$ and that $\sigma_u = \prod_{i \in u} y_i$ for all $\emptyset \neq u \subseteq \mathcal{S}$ for some numbers $y_i, z_i \ge 0$.

If the weights $\gamma_u$ are uniformly distributed in the interval $[\overline{\gamma_u} + \delta_u, \overline{\gamma}_u - \delta_u]$, where $0 \le \delta_u \le \overline{\gamma}_u$, then the expectation value is $\overline{\gamma}_u \ge 0$ and $\mathbb{E}(\gamma_u^2) = \frac{1}{2\delta_u} \int_{\overline{\gamma}_u-\delta_u}^{\overline{\gamma}_u+\delta_u} \gamma_u^2 \,\mathrm{d} \gamma_u = \overline{\gamma}_u^2 + \delta_u^2/3$, thus the variance is $\gamma_u^2 = \delta_u^2/3$. Note that $\overline{\gamma}_u - \delta_u \ge 0$, which ensures that the weights are always non-negative. In the following we do not assume that the weights are uniformly distributed, in fact, it is more interesting to assume a different distribution which also allows weights much larger than $\overline{\gamma}_u + \delta_u$.

In the following the expectation $\mathbb{E}$, the variance $\mathrm{Var}$ and the standard deviation $\mathrm{Std}$ are always taken with respect to the random variables $\gamma_u$. The expectation value is now
\begin{equation}\label{eq_expectation}
\mathbb{E}( e^2(\mathcal{H}_{\boldsymbol{\gamma}},P_N(\boldsymbol{g})) ) =  \sum_{\emptyset \neq u \subseteq \mathcal{S}} \overline{\gamma}_u \frac{1}{N} \sum_{n=0}^{N-1}\prod_{i \in u} B_2(\{n g_i/N\}) = e^2(\mathcal{H}_{\overline{\boldsymbol{\gamma}}}, P_N(\boldsymbol{g})).
\end{equation}
Thus, current construction algorithms \cite{NC06,NC06b,SKJ02,SKJ03} can be viewed as finding quadrature rules for which the expected value is small. Here we aim at finding quadrature rules for which, additionally, the variance is small.

We point out that there is a difference between $\overline{\gamma}_u$ 'very small' and $\overline{\gamma}_u = 0$. The restriction $\gamma_u \ge 0$ implies that if $\overline{\gamma}_u = 0$ then $\sigma_u = 0$. Thus if one constructs a lattice rule with error criterion $e^2(\mathcal{H}_{\overline{\boldsymbol{\gamma}}},P_N(\boldsymbol{g}))$ where $\overline{\gamma}_u =0$ for some $\emptyset \neq u \subseteq \mathcal{S}$, then this means that the random variable $\gamma_u = 0$ with probability $1$. Thus setting $\overline{\gamma}_u=0$ means that one knows that $\gamma_u = 0$ and the associated ANOVA term $f_u =0$. Lattice rules constructed using such weights do not have any guarantee that positive weights $\gamma_u > 0$ will yield a good result. To illustrate, consider the two-dimensional example where $\gamma_{1}=1$, $\gamma_{2}=1$ and $\gamma_{1,2}=0$ and the lattice rule has generating vector $\boldsymbol{g}=(1,1)$. This lattice rule works well in this case but not if the weight $\gamma_{1,2}$ changes to $1$, say; see also \cite{S07}.

Using some elementary properties of the variance we obtain
\begin{align*}
\mathrm{Var}(e^2(\mathcal{H}_{\boldsymbol{\gamma}}, P_N(\boldsymbol{g})))  & = \sum_{\emptyset \neq u \subseteq \mathcal{S}} \left(\sigma_u \frac{1}{N} \sum_{n=0}^{N-1} \prod_{i \in u} B_2(\{n g_i/N\}) \right)^2
\end{align*}
and the standard deviation is given by
\begin{equation*}
\mathrm{Std}(e^2(\mathcal{H}_{\boldsymbol{\gamma}}, P_N(\boldsymbol{g}))) = \sqrt{\mathrm{Var}(e^2(\mathcal{H}_{\boldsymbol{\gamma}}, P_N(\boldsymbol{g}))) } = \sqrt{\sum_{\emptyset \neq u \subseteq \mathcal{S}} \left(\sigma_u \frac{1}{N} \sum_{n=0}^{N-1} \prod_{i \in u} B_2(\{n g_i/N\}) \right)^2}.
\end{equation*}
Lattice rules for which the variance $\mathrm{Var}(e^2(\mathcal{H}_{\boldsymbol{\gamma}}, P_N(\boldsymbol{g}))) $ is small are less sensitive to changes of the weights $\gamma_u$. Notice that the variance is difficult to compute in general in high dimensions since it involves a sum over all subsets of $\{1,\ldots, s\}$ for which $\sigma_u > 0$ and which, in general, cannot easily be simplified to a formula which can be computed quickly, even in the case where the weights are of product form.

Notice that \eqref{eq_expectation} is the one-norm of the vector consisting of the error of the projections weighted by the expectation values of the weights, whereas the standard deviation is the two-norm of the vector consisting of the error of the projections weighted by the variance of the weights.

Let $\mu$ be a probability measure on the weights $(\gamma_u)_{\emptyset \neq u \subseteq \mathcal{S}}$. We now use the one-sided Chebyshev inequality which states that for a random variable $X$ with probability measure $\mathrm{Pr}$, expectation $\mathbb{E}(X)$ and standard deviation $\mathrm{Std}(X)$, we have for any $c > 0$ that
\begin{equation*}
\mathrm{Pr}( X - \mathbb{E}(X) \ge c \, \mathrm{Std}(X)) \ge \frac{1}{1+c^2}.
\end{equation*}
Thus we obtain the following result.
\begin{lemma}\label{lem_cheb}
For any $c > 0$ we have
\begin{equation*}
\mu\left(\boldsymbol{\gamma}: e^2(\mathcal{H}_{\boldsymbol{\gamma}}, P_N(\boldsymbol{g}))  \le e^2(\mathcal{H}_{\overline{\boldsymbol{\gamma}}}, P_N(\boldsymbol{g})) + c \, \mathrm{Std}(e^2(\mathcal{H}_{\boldsymbol{\gamma}},P_N(\boldsymbol{g}))) \right) \ge \frac{c^2}{1+c^2}.
\end{equation*}
\end{lemma}

As noted above, the standard deviation is in general difficult to compute, however, using Jensen's inequality we have
\begin{align*}
\mathrm{Std}(e^2(\mathcal{H}_{\boldsymbol{\sigma}}, P_N(\boldsymbol{g}))) & = \sqrt{\sum_{\emptyset \neq u \subseteq \mathcal{S}} \left(\sigma_u \frac{1}{N} \sum_{n=0}^{N-1} \prod_{i \in u} B_2(\{n g_i/N\}) \right)^2} \\ & \le \sum_{\emptyset \neq u \subseteq \mathcal{S}} \sigma_u \frac{1}{N} \sum_{n=0}^{N-1} \prod_{i\in u} B_2(\{n g_i/N\}) = e^2(\mathcal{H}_{\boldsymbol{\sigma}}, P_N(\boldsymbol{g})).
\end{align*}
Thus the square worst-case error with the variances as weights is an upper bound on the standard deviation. This upper bound can easily be computed (for instance for variances of product form).

Assume now that if for some $\emptyset \neq u \subseteq \mathcal{S}$ we have $\sigma_u  = 0$, then also $\gamma_u = 0$ and therefore $\overline{\gamma}_u = 0$. Using H\"older's inequality we have
\begin{align*}
e^2(\mathcal{H}_{\boldsymbol{\gamma}}, P_N(\boldsymbol{g})) & = \sum_{\satop{\emptyset \neq u \subseteq \mathcal{S}}{\sigma_u > 0}} \frac{\gamma_u}{\sigma_u} \left(\sigma_u \frac{1}{N} \sum_{n=0}^{N-1} \prod_{i \in u} B_2(\{n g_i/N\}) \right) \\ & \le \mathrm{Std}(e^2(\mathcal{H}_{\boldsymbol{\sigma}}, P_N(\boldsymbol{g}))) \sqrt{\sum_{\satop{\emptyset \neq u \subseteq \mathcal{S}}{\sigma_u > 0}} \frac{\gamma_u^2}{\sigma_u^2} }.
\end{align*}
Thus a small standard deviation implies a small expected error. Combining the last two inequalities we obtain the following result.
\begin{lemma}
Assume that if for some $\emptyset \neq u \subseteq \mathcal{S}$ we have $\sigma_u  = 0$, then also $\gamma_u = 0$ and therefore $\overline{\gamma}_u = 0$. Then we have
\begin{equation*}
e^2(\mathcal{H}_{\boldsymbol{\gamma}}, P_N(\boldsymbol{g})) \le \mathrm{Std}(e^2(\mathcal{H}_{\boldsymbol{\sigma}}, P_N(\boldsymbol{g}))) \sqrt{\sum_{\emptyset \neq u \subseteq \mathcal{S}} \frac{\gamma_u^2}{\sigma_u^2}} \le e^2(\mathcal{H}_{\boldsymbol{\sigma}},P_N(\boldsymbol{g})) \sqrt{\sum_{\emptyset \neq u \subseteq \mathcal{S}} \frac{\gamma_u^2}{\sigma_u^2} }.
\end{equation*}
\end{lemma}

If the weights $\boldsymbol{\sigma}$ are decaying such that  the upper bound is independent of the dimension, i.e. strong tractability (see \cite{NW08,NW10,SW98}) holds, then for any weights $\boldsymbol{\gamma}$ such that the expression $\sum_{u\subset \mathbb{N}, |u| <\infty} \frac{\gamma_u^2}{\sigma_u^2} < \infty$ one also obtains a bound which is independent of the dimension, i.e. one has strong tractability. Further, if $e^2(\mathcal{H}_{\boldsymbol{\sigma}},P_N(\boldsymbol{g}))$ satisfies strong tractability, then also the standard deviation is bounded.

Combining the last two lemmas we obtain the following corollary.
\begin{corollary}
Assume that if for some $\emptyset \neq u \subseteq \mathcal{S}$ we have $\sigma_u  = 0$, then also $\gamma_u = 0$ and therefore $\overline{\gamma}_u = 0$. Then for any $c > 0$ we have
\begin{equation*}
\mu\left(\boldsymbol{\gamma}: e^2(\mathcal{H}_{\boldsymbol{\gamma}}, P_N(\boldsymbol{g}))  \le e^2(\mathcal{H}_{\boldsymbol{\sigma}}, P_N(\boldsymbol{g})) \left(\sqrt{\sum_{\satop{\emptyset \neq u \subseteq \mathcal{S}}{\sigma_u > 0}} \overline{\gamma}^2_u/\sigma_u^2 } + c \right) \right) \ge \frac{c^2}{1+c^2}.
\end{equation*}
\end{corollary}
This result can be viewed as a robustness result with respect to weights. If one constructs a lattice rule using $\boldsymbol{\sigma} = (\sigma_u)$ as weights, then for a set of weights taken from a set with measure at least $\frac{c^2}{1+c^2}$, the error is bounded by
\begin{equation*}
e^2(\mathcal{H}_{\boldsymbol{\sigma}}, P_N(\boldsymbol{g})) \left(\sqrt{\sum_{\satop{\emptyset \neq u \subseteq \mathcal{S}}{\gamma_u > 0}} \overline{\gamma}^2_u/\sigma_u^2 } + c \right).
\end{equation*}

There is one notable exception, namely, if $\sigma_u =0$ for some $\emptyset \neq u \subseteq \mathcal{S}$, then no robustness with respect to the projection on $u$ can be obtained (as was also illustrated above).

In the following we show that one can do better by taking the robustness into account in the construction of the lattice rule itself.

\section{A construction of robust lattice rules}\label{sec3}

In this section we generalize the component-by-component algorithm from \cite{K59,SR02}. We repeat some facts from the fast cbc algorithm of Nuyens and Cools~\cite{NC06, NC06b}.

\subsection{The fast Fourier transform method}

Nuyens and Cools \cite{NC06,NC06b} have shown how to use the fast Fourier transform to reduce the computation time of the component-by-component algorithm. Because of the importance of these ideas we repeat them here (as is well understood, we see below that the algorithm of Nuyens and Cools actually calculates slightly more, which is important for the cbc$r$c algorithm below).

For simplicity of exposition we assume product weights $\gamma_u = \prod_{i\in u} \widehat{\gamma}_i$ and that $N$ is a prime number. For more general cases see \cite{NC06,NC06b}. Assume that the coordinates $g_1^\ast,\ldots, g_{j-1}^\ast \in \{1,\ldots, N-1\}$ are already fixed. Then we write the error criterion for $j > 1$ in the form
\begin{align*}
& e^2(\mathcal{H}_{\boldsymbol{\gamma}},P_N((g^\ast_1,\ldots, g^\ast_{j-1},z))) \\ & = -1 + \frac{1}{N} \sum_{n=0}^{N-1} \prod_{i=1}^{j-1} \left(1 + \widehat{\gamma}_i B_2(\{n g_i^\ast /N\})\right) \left(1+ \widehat{\gamma}_j B_2(\{n z/N\})\right) \\ & = -1 + \frac{1}{N} \sum_{n=0}^{N-1} \prod_{i=1}^{j-1} \left(1 + \widehat{\gamma}_i B_2(\{n g_i^\ast /N\})\right) + \frac{1}{N} \sum_{n=0}^{N-1} \prod_{i=1}^{j-1} \left(1 + \widehat{\gamma}_i B_2(\{n g_i^\ast / N\}) \right) B_2(\{n z / N\}) \\ & = e^2(\mathcal{H}_{\boldsymbol{\gamma}}, P_N((g_1^\ast,\ldots, g_{j-1}^\ast))) + \frac{1}{N}  \prod_{i=1}^{j-1} \left(1 + \widehat{\gamma}_i / 6 \right) / 6 \\ & + \frac{1}{N} \sum_{n=1}^{N-1} \prod_{i=1}^{j-1} \left(1 + \widehat{\gamma}_i B_2(\{n g_i^\ast / N\}) \right) B_2(\{n z / N\}).
\end{align*}
Since the components $g_1^\ast,\ldots, g_{j-1}^\ast$ are fixed, the value $e^2(\mathcal{H}_{\boldsymbol{\gamma}}, P_N((g_1,\ldots, g_{j-1})))$ and $\frac{1}{N}  \prod_{i=1}^{j-1} \left(1 + \widehat{\gamma}_i / 6 \right) / 6 $ does not depend on $z$ and can therefore be ignored. Thus it suffices to calculate
\begin{equation*}
\Psi(z) = \sum_{n=1}^{N-1} \prod_{i=1}^{j-1} \left(1 + \widehat{\gamma}_i B_2(\{n g_i^\ast / N\}) \right) B_2(\{n z / N\}) \quad \mbox{for } 1 \le z \le N-1.
\end{equation*}
We define the matrix
\begin{equation*}
\Omega = \left(B_2(\{n z/ N\})\right)_{1 \le n,z \le N-1}
\end{equation*}
and the vector
\begin{equation*}
\boldsymbol{p} = \left(\prod_{i=1}^{j-1} (1+\widehat{\gamma}_i B_2(\{0 g_i^\ast/N\}) ), \ldots, \prod_{i=1}^{j-1} (1 + \widehat{\gamma}_{i} B_2(\{(N-1) g_{i}^\ast/N\}) ) \right)^\top.
\end{equation*}
Then
\begin{equation*}
(\Psi(1),\ldots, \Psi(N-1))^\top = \Omega \boldsymbol{p}.
\end{equation*}
The matrix $\Omega$ has some structure which allows one to use the fast Fourier transform. Let $1 < v < N$ be a primitive element in the finite field $\mathbb{Z}_N=\{0,1,\ldots, N-1\}$ of prime order $N$. Note that the multiplicative inverse $v^{-1}$ is then also a primitive element. We define the permutation matrix $\Pi(v) = (\pi_{k,l}(v))_{1 \le k,l \le N-1}$ by
\begin{equation*}
\pi_{k,l}(v) = \left\{\begin{array}{ll} 1 & \mbox{if } k = v^l \pmod{N}, \\ 0 & \mbox{otherwise}. \end{array} \right.
\end{equation*}
Note that $\Pi(v) \Pi(v)^\top = I$, the identity matrix. Let $C=(c_{k,l})_{1 \le k,l \le N-1}$ be defined by
\begin{equation*}
C = \Pi(v)^\top \Omega \Pi(v^{-1}),
\end{equation*}
hence
\begin{equation*}
c_{k,l} = \sum_{u,w=1}^{N-1} \pi_{u,k}(v) B_2(\{uw/N\}) \pi_{w,l}(v^{-1}) = B_2(\{v^{k-l}/N\}).
\end{equation*}
The matrix $C$ is therefore circulant. Let $F_{N-1} = (N-1)^{-1/2} (f_{k,l})_{0 \le k,l \le N-2}$ be the Fourier matrix of order $N-1$ where $f_{k,l} = \mathrm{e}^{2\pi \mathrm{i} kl/N}$. Then $D = F_{N-1} C F_{N-1}^{-1}$ is a diagonal matrix. Thus we have
\begin{equation*}
\Omega = \Pi(v) C \Pi(v^{-1})^\top = \Pi(v) F_{N-1}^{-1} D F_{N-1} \Pi(v^{-1})^\top.
\end{equation*}
Consider now the matrix-vector multiplication $\Omega \boldsymbol{p}$. Multiplying a vector with the permutation matrices $\Pi(v), \Pi(v^{-1})^\top$ takes $\mathcal{O}(N)$ operations, the matrix vector-multiplication with the matrices $F_{N-1}, F_{N-1}^{-1}$ can be carried out in $\mathcal{O}(N\log N)$ operations using the fast Fourier transform. Multiplying the diagonal matrix $D$ with a vector takes $\mathcal{O}(N)$ operations. Thus the matrix-vector multiplication $\Omega \boldsymbol{p}$  can be carried out in $\mathcal{O}(N \log N)$ operations. For more details see \cite{NC06c}.

Notice that the fast matrix vector multiplication directly yields the whole vector $(\Psi(1),\ldots, \Psi(N-1))^\top$. This vector can be ordered (using a sorting algorithm) to obtain $\Psi(z_1) \le \Psi(z_2) \le \cdots \le \Psi(z_{N-1})$. This can be done in $\mathcal{O}(N \log N)$ operations. Thus, by the above arguments, we can compute $z_1,\ldots, z_{N-1}$ such that $$e^2(\mathcal{H}_{\boldsymbol{\gamma}}, P_N((g_1^\ast,\ldots, g_{j-1}^\ast,z_1))) \le  \cdots \le e^2(\mathcal{H}_{\boldsymbol{\gamma}}, P_N((g_1^\ast,\ldots, g_{j-1}^\ast,z_{N-1})))$$ in $\mathcal{O}(N \log N)$ operations.

\subsection{The fast cbc$2$c algorithm}

In the previous section we have shown some robustness of lattice rules which are constructed for a given set of weights. In this section we modify the fast cbc algorithm \cite{NC06,NC06b} to construct lattice rules for which, simultaneously, $\mathbb{E}(e^2(\mathcal{H}_{\boldsymbol{\gamma}}, P_N(\boldsymbol{g}))) = e^2(\mathcal{H}_{\overline{\boldsymbol{\gamma}}}, P_N(\boldsymbol{g}))$ and  $\mathrm{Std}(e^2(\mathcal{H}_{\boldsymbol{\gamma}}, P_N(\boldsymbol{g})))$ are small. Since the standard deviation $\mathrm{Std}(e^2(\mathcal{H}_{\boldsymbol{\gamma}}, P_N(\boldsymbol{g})))$ is in general difficult to compute, we use $e^2(\mathcal{H}_{\boldsymbol{\sigma}}, P_N(\boldsymbol{g}))$ as criterion instead. This has the additional advantage that the roles of $\overline{\boldsymbol{\gamma}}$ and $\boldsymbol{\sigma}$ are interchangeable.

Throughout the paper let $\kappa$ denote the number of distinct prime factors of the integer $N \ge 2$. We use \cite[Theorem~3]{DPW08}, which states that for any $c \ge 1$, the proportion of generating vectors $\boldsymbol{g} \in \{1,\ldots, N-1\}^s$ which satisfy
\begin{equation}\label{ineq_bound}
e^2(\mathcal{H}_{\boldsymbol{\gamma}}, P_N(\boldsymbol{g})) \le  \left( \frac{c}{N}\sum_{\emptyset \neq u \subseteq \mathcal{S}} \gamma_u^{1/\tau} (2^\kappa \pi^{-2} \zeta(2 /\tau) )^{|u|} \right)^{\tau} \quad \mbox{for all } 1 \le \tau < 2,
\end{equation}
where $\zeta(r) = \sum_{k=1}^\infty k^{-r}$ is the Riemann zeta function, is bigger than $1-c^{-1}$, i.e. there are more than $(N-1)^s (1-c^{-1})$ generating vectors $\boldsymbol{g} \in \{1,\ldots, N-1\}^s$ which satisfy the above bound (see also \cite{SL11} for other criteria and bounds when $N$ is not prime). Further, \cite[Theorem~10]{DPW08} states that a generating vector $\boldsymbol{g}^\ast = (g_1^\ast,\ldots, g_s^\ast) \in \{1,\ldots, N-1\}$ which satisfies \eqref{ineq_bound} can be found component-by-component. Thus we obtain the following result which follows from the fact that the intersection of two sets with measure bigger than $1-c_1^{-1}$ and $1-c_2^{-1}$, where $c_1,c_2 \ge 1$ are such that $1-c^{-1}_1 + 1-c_{2}^{-1} \ge 1$, is non-empty.

\begin{algorithm}[The fast component-by-component two criteria (fast cbc$2$c algorithm]\label{alg1}
Given: natural numbers $N, s$, nonnegative real numbers $\overline{\gamma}_u, \sigma_u$ for all $\emptyset \neq u \subseteq \mathcal{S}$;  $1 \le c_1 \le 2$, $2 \le c_2 \le \infty$ such that $c_1^{-1} + c_2^{-1} = 1$.
\begin{itemize}
\item Set $g_1^\ast = 1$;
\item For $j=2,\ldots, s$ do the following:
\begin{itemize}
\item  {\bf Hard constraint}

Let $K_1 =  \min\{\lfloor (N-1) (1-c^{-1}_1) \rfloor+1, N-1\}$. Find the set of integers $A = \{z_1,\ldots, z_{K_1} \} \subseteq \{1,\ldots, N-1\}$ which satisfies:
\begin{equation*}
e^2(\mathcal{H}_{\overline{\boldsymbol{\gamma}}}, P_N((g_1^\ast, \ldots, g_{j-1}^\ast, z_i)))  \le e^2(\mathcal{H}_{\overline{\boldsymbol{\gamma}}}, P_N((g_1^\ast, \ldots, g_{j-1}^\ast, z)))
\end{equation*}
for all $1 \le i \le K_1$ and $z \in \{1,\ldots, N-1\} \setminus A$ (using the fast Fourier transform).

\item {\bf Soft constraint}

Let $K_2 =  \min\{\lfloor (N-1) (1-c^{-1}_2) \rfloor + 1, N-1\}$. Find the set of integers $B = \{y_1,\ldots, y_{K_2}\} \subseteq \{1,\ldots, N-1\}$ which satisfies:
\begin{equation*}
e^2(\mathcal{H}_{\boldsymbol{\sigma}}, P_N((g_1^\ast, \ldots, g_{j-1}^\ast, y_i)))  \le e^2(\mathcal{H}_{\boldsymbol{\sigma}}, P_N((g_1^\ast, \ldots, g_{j-1}^\ast, y)))
\end{equation*}
for all $1 \le i \le K_2$ and $y \in \{1,\ldots, N-1\} \setminus B$ (using the fast Fourier transform).

\item Choose $g_j^\ast \in A \cap B$ which minimizes $e^2(\mathcal{H}_{\overline{\boldsymbol{\gamma}}}, P_N((g_1^\ast, \ldots, g_{j-1}^\ast, w)))$ as a function of $w$.
\end{itemize}
\item Return $\boldsymbol{g}^\ast = (g_1^\ast, \ldots, g_s^\ast)$.
\end{itemize}
\end{algorithm}

Some comments are in order:
\begin{itemize}
\item[(i)] We have $K_1 + K_2 > (N-1) (1-c_1^{-1} + 1-c_2^{-1}) = N-1$. Thus the set $A \cap B$ is not empty. Further note that the vector $\boldsymbol{g}^\ast$ found by Algorithm~\ref{alg1} satisfies the bounds in Theorem~\ref{thm1} (see also  \cite[Theorem~10]{DPW08}).
\item[(ii)] The algorithm is basically symmetric in the constraints, but, by choosing $1 \le c_1 \le 2$, the first constraint is at least as hard to satisfy as the second one, since the upper bound is lower.
\item[(iii)] We have biased the algorithm towards the hard constraint. Instead of choosing the value $g_d^\ast \in A \cap B$ which minimizes $e^2(\mathcal{H}_{\overline{\boldsymbol{\gamma}}}, P_N((g_1^\ast, \ldots, g_{d-1}^\ast, w)))$ actually any value in the set $A \cap B$ could be chosen. The results still apply in this case.
\item[(iv)] The classical cbc algorithm corresponds to the special case where $c_1=1$ and $c_2=\infty$. Further, the classical cbc algorithm can also be obtained by choosing $\overline{\boldsymbol{\gamma}} = \boldsymbol{\sigma}$ (in which case the choice of $1 \le c_1, c_2 \le \infty$ is irrelevant).
\item[(v)] The fast cbc algorithm can be used to calculate the values $e^2(\mathcal{H}_{\overline{\boldsymbol{\gamma}}}, P_N((g_1^\ast, \ldots, g_{d-1}^\ast, z)))$ and $e^2(\mathcal{H}_{\boldsymbol{\sigma}}, P_N((g_1^\ast, \ldots, g_{d-1}^\ast, y)))$ for all $y,z \in \{1,\ldots, N-1\}$ very efficiently. The values need to be sorted and then one needs to choose a value in $A \cap B$. The main complexity is calculating the worst-case errors, hence the number of operations needed for the algorithm is the same as that for the fast cbc construction \cite{NC06,NC06b} (but with a larger constant since we have two worst-case errors). Thus one has a fast cbc2c algorithm.
\end{itemize}

From \cite[Theorem~10]{DPW08} we obtain the following result concerning the cbc$2$c algorithm.

\begin{theorem}\label{thm1}
Let $N$ be an integer and $1 \le c_1,c_2 \le \infty$ such that $c_1^{-1} + c_2^{-1} = 1$. Then the generating vector $\boldsymbol{g}^\ast \in \{1,\ldots, N-1\}^s$ constructed by the cbc$2$c algorithm satisfies
\begin{align*}
e^2(\mathcal{H}_{\overline{\boldsymbol{\gamma}}}, P_N(\boldsymbol{g}^\ast)) & \le \left(\frac{c_1}{N} \sum_{\emptyset \neq u \subseteq \mathcal{S}} \overline{\gamma}_u^{1/\tau} (2^\kappa\pi^{-2} \zeta(2/\tau))^{|u|} \right)^{\tau} \quad \mbox{for all } 1 \le \tau < 2, \mbox{ and } \\ e^2(\mathcal{H}_{\boldsymbol{\sigma}}, P_N(\boldsymbol{g}^\ast))) & \le \left(\frac{c_2}{N} \sum_{\emptyset \neq u \subseteq \mathcal{S}} \sigma_u^{1/\tau} (2^\kappa\pi^{-2} \zeta(2/\tau))^{|u|} \right)^{\tau} \quad \mbox{for all } 1 \le \tau < 2.
\end{align*}
\end{theorem}

For $c > 0$ and $1 \le \tau, \tau' < 2$ let
\begin{equation}\label{def_B}
B_{\overline{\boldsymbol{\gamma}}, \boldsymbol{\sigma}}(c,\tau,\tau',N) = \left(\frac{c_1}{N} \sum_{\emptyset \neq u \subseteq \mathcal{S}} \overline{\gamma}_u^{1/\tau} (2^\kappa \pi^{-2} \zeta(2/\tau))^{|u|} \right)^{\tau}  + c \left(\frac{c_2}{N} \sum_{\emptyset \neq u \subseteq \mathcal{S}} \sigma_u^{1/\tau'} (2^\kappa \pi^{-2} \zeta(2/\tau'))^{|u|} \right)^{\tau'}.
\end{equation}

We obtain the following corollary from Lemma~\ref{lem_cheb} and Theorem~\ref{thm1}.

\begin{corollary}\label{cor3}
The generating vector $\boldsymbol{g}^\ast = (g_1^\ast,\ldots, g_s^\ast)$ constructed by the cbc$2$c algorithm satisfies
\begin{equation*}
\mu\left(\boldsymbol{\gamma}: e^2(\mathcal{H}_{\boldsymbol{\gamma}}, P_N(\boldsymbol{g}^\ast)) \le B_{\overline{\boldsymbol{\gamma}}, \boldsymbol{\sigma}}(c,\tau,\tau',N) \mbox{ for all } 1 \le \tau, \tau' < 2 \right) \ge \frac{c^2}{1+c^2} \quad \mbox{for any } c > 0.
\end{equation*}
\end{corollary}

By choosing $\overline{\boldsymbol{\gamma}} = \boldsymbol{\sigma}$, Algorithm~\ref{alg1} can be simplified to the classical cbc algorithm. Thus Corollary~\ref{cor3}, with $\overline{\boldsymbol{\gamma}} = \boldsymbol{\sigma}$, applies to the classical fast cbc algorithm. However, in this case, if $\overline{\gamma}_u$ is small then also $\sigma_u = \overline{\gamma}_u$ is small and thus the lattice rule may be sensitive to changes in the projection $u$. Unfortunately Corollary~\ref{cor3} does not give any information about the set of the weights which satisfy the condition. We study this topic in the following more general setting of the cbc algorithm with $r$ constraints.

\subsection{The fast cbc$r$c algorithm}

In this subsection, instead of two constraints we study a cbc algorithm using $1 \le r \le 2^s-1$ constraints. In this case one needs $r$ sets of weights $\boldsymbol{\gamma}^{(1)}, \ldots, \boldsymbol{\gamma}^{(r)}$ which are linearly independent in $\mathbb{R}^{2^s-1}$ and  $1 \le c_1, \ldots, c_r \le \infty$ such that $c_1^{-1} + \cdots + c_r^{-1} = 1$. As we will see below, adding a vector of weights $\boldsymbol{\gamma}'$ which is a linear combination of the weights $\boldsymbol{\gamma}^{(1)},\ldots, \boldsymbol{\gamma}^{(r)}$ does not add a new constraint, since any vector satisfying the first $r$ constraints automatically satisfies the constraint using the weight $\boldsymbol{\gamma}'$.

\begin{algorithm}[The fast component-by-component $r$ criteria (fast cbc$r$c) algorithm]\label{alg2}
Given:  natural numbers $N, s$ and $1 \le r \le 2^s-1$, nonnegative real vectors $\boldsymbol{\gamma}^{(1)}=(\gamma_u^{(1)})_{\emptyset \neq u \subseteq \mathcal{S}}, \ldots, \boldsymbol{\gamma}^{(r)} = (\gamma_u^{(r)})_{\emptyset \neq u \subseteq \mathcal{S}}$ which are linearly independent in $\mathbb{R}^{2^s-1}$; positive numbers $c_1,\ldots, c_r \in \mathbb{R} \cup \{\infty\}$ which satisfy $c_r \ge  \cdots \ge c_1 \ge 1$ and $c_1^{-1} + \cdots + c_r^{-1} = 1$.
\begin{itemize}
\item Set $K_w =  \min\{\lfloor (N-1) (1-c_w^{-1}) \rfloor + 1, N-1\}$ for $1 \le w \le r$.
\item Set $g_1^\ast = 1$;
\item For $j=2,\ldots, s$ do the following:
\begin{itemize}
\item For $w=1,\ldots, r$ do the following:

Find the set of integers $A_w = \{z^{(w)}_1,\ldots, z^{(w)}_{K_w} \} \subseteq \{1,\ldots, N-1\}$ which satisfies:
\begin{equation*}
e^2(\mathcal{H}_{\boldsymbol{\gamma}^{(w)}}, P_N((g_1^\ast, \ldots, g_{j-1}^\ast, z^{(w)}_i)))  \le e^2(\mathcal{H}_{\boldsymbol{\gamma}^{(w)}}, P_N((g_1^\ast, \ldots, g_{j-1}^\ast, z)))
\end{equation*}
for all $1 \le i \le K_w$ and $z \in \{1,\ldots, N-1\} \setminus A_w$ using the fast algorithm described above.

\item Choose $g_d^\ast \in \bigcap_{1\le w \le r} A_w$ which minimizes $e^2(\mathcal{H}_{\boldsymbol{\gamma}^{(1)}}, P_N((g_1^\ast, \ldots, g_{d-1}^\ast, z)))$ as a function of $z$.
\end{itemize}
\item Return $\boldsymbol{g}^\ast = (g_1^\ast, \ldots, g_s^\ast)$.
\end{itemize}
\end{algorithm}

The considerations above imply therefore that the construction cost of the fast cbc$r$c algorithm is $\mathcal{O}(r s N (\log N))$ operations using $\mathcal{O}(r N)$ storage (note that the intersection step can be done by sorting the elements in $A_1,\ldots, A_r$ first, which takes $\mathcal{O}(N \log N)$ operations).

From \cite[Theorem~10]{DPW08} we also obtain a generalization of Theorem~\ref{thm1} which applies to the cbc$r$c algorithm.

\begin{theorem}\label{thm1b}
Let $N \ge 2$ and $r \ge 1$ be integers and $1 \le c_1, \ldots, c_r \le \infty$ such that $c_1^{-1} + \cdots + c_r^{-1} = 1$. Then the generating vector $\boldsymbol{g}^\ast \in \{1,\ldots, N-1\}^s$ constructed by the cbc$r$c algorithm using the weights $\boldsymbol{\gamma}^{(1)},\ldots, \boldsymbol{\gamma}^{(r)}$ satisfies
\begin{equation*}
e^2(\mathcal{H}_{\boldsymbol{\gamma}^{(w)}}, P_N(\boldsymbol{g}^\ast)) \le \left(\frac{c_w}{N} \sum_{\emptyset \neq u \subseteq \mathcal{S}} (\gamma^{(w)}_u)^{1/\tau} (2^\kappa\pi^{-2} \zeta(2/\tau))^{|u|} \right)^{\tau} \quad \mbox{for all } 1 \le \tau < 2
\end{equation*}
and all $1 \le w \le r$.
\end{theorem}

\subsection{The geometry of the cbc$r$c algorithm}\label{sec_geom_alg}

For a $\boldsymbol{x} \in \mathbb{R}^{2^s-1}$ we write $\boldsymbol{x} \le \boldsymbol{y}$ if $x_i \le y_i$ for all $1 \le i < 2^s$, where $\boldsymbol{x} = (x_1,\ldots, x_{2^s-1})$ and $\boldsymbol{y}=(y_1,\ldots, y_{2^s-1})$. Similarly we use the symbols $<, >, \ge$.

For $\boldsymbol{z} = (z_1,\ldots, z_{2^s-1}) \in \mathbb{R}^{2^s-1}$ and $\varepsilon > 0$ we define the simplex
\begin{equation*}
\Gamma(\boldsymbol{z}, \varepsilon)  = \{\boldsymbol{y} \in \mathbb{R}^{2^s-1}: \boldsymbol{y} \ge \boldsymbol{0}, \boldsymbol{y} \cdot \boldsymbol{z} \le \varepsilon\}.
\end{equation*}
This simplex has vertices $\boldsymbol{0}$ and $(\boldsymbol{0}, \varepsilon/z_i)$, $1 \le i < 2^s$, which stands for the vector whose $i'$th component is $0$ for $i \neq i'$ and whose $i$th component is $z_i$. If $z_i = 0$ for some component $i$, then the simplex is degenerate and we consider the projection of the set onto those components which are nonzero (which is then a nondegenerate simplex).

Let $\boldsymbol{g} = (g_1,\ldots, g_s)$, $\boldsymbol{g}_u = (g_i)_{i \in u}$ and
\begin{align*}
e_u & = e_u(\boldsymbol{g}_u) = \frac{1}{N} \sum_{n=0}^{N-1} \prod_{i \in u} B_2(\{n g_i/N\}) \quad \mbox{for } \emptyset \neq u \subseteq \mathcal{S}, \\
\boldsymbol{e} & = \boldsymbol{e}(\boldsymbol{g}) = (e_u)_{\emptyset \neq u \subseteq \mathcal{S}}, \\
X_{\boldsymbol{\gamma},c} & = X_{\boldsymbol{\gamma},c}(N) = \inf_{1\le \tau < 2} \left(\frac{c}{N} \sum_{\emptyset \neq u \subseteq \mathcal{S}} (\gamma_u)^{1/\tau} (2 \pi^{-2} \zeta(2/\tau))^{|u|} \right)^\tau.
\end{align*}
Notice that since $e_u$ is the worst-case error of integration in a reproducing kernel Hilbert space, we have $e_u \ge 0$ and therefore $\boldsymbol{e} \ge \boldsymbol{0}$.

Theorem~\ref{thm1b} implies that the cbc$r$c algorithm now chooses the generating vector $\boldsymbol{g} \in \{1,\ldots, N-1\}^s$ such that
\begin{equation*}
\boldsymbol{\gamma}^{(w)} \cdot \boldsymbol{e} := (\gamma^{(w)}_u) \cdot (e_u) : = \sum_{\emptyset \neq u \subseteq \mathcal{S}} \gamma^{(w)}_u e_u \le X_{\boldsymbol{\gamma}^{(w)}, c_w} \quad \mbox{for } 1 \le w \le r.
\end{equation*}
Thus $\boldsymbol{e}$ lies in the intersection of the simplices $\Gamma(\boldsymbol{\gamma}^{(w)}, X_{\boldsymbol{\gamma}^{(w)},c_w})$ for $1 \le w \le r$:
\begin{equation}\label{eq_simplex}
\boldsymbol{e} \in \bigcap_{1 \le w \le r} \Gamma(\boldsymbol{\gamma}^{(w)}, X_{\boldsymbol{\gamma}^{(w)},c_w}).
\end{equation}
Geometrically this means that $\boldsymbol{e}$ lies in a convex $r$-polytope given by the intersection of $r$ simplices. The weights $\boldsymbol{\gamma}^{(w)}$ change the shape of the simplices, whereas the values $X_{\boldsymbol{\gamma}^{(w)}, c_w}$ change the size of the simplices.

Compared to the cbc algorithm, the component-by-component $r$ criteria (cbc$r$c) algorithm first increases the original simplex (by at most a factor of $c_r$) and then intersect it with other simplices, see Figure~\ref{fig1}. This can be used to prevent $\boldsymbol{e}$ to be chosen too close to a vertex of the original simplex (this prevents $e_u$ from becoming too large for some $\emptyset \neq u \subseteq \mathcal{S}$).

\begin{figure}[ht]
\begin{center} 
\includegraphics[scale= 0.5]{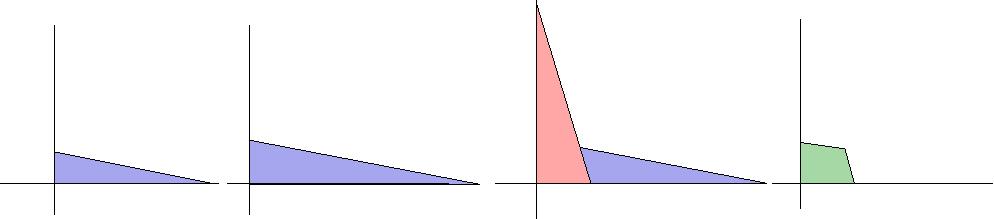}
\caption{\label{fig1} The figure shows the search domains. The classical cbc algorithm searches in a simplex as indicated in the left picture. The cbc2c algorithm first increases the size of the simplex (second picture) and then intersects it with another simplex (third picture) to get the new search domain (fourth picture). This way extreme choices near the vertex of the original simplex (which corresponds to a large value of $e_u$ for some $u$) can be avoided.}
\end{center}
\end{figure}

\subsection{The geometry of the weights}

We now study the geometry of the weights for which the corresponding square worst-case error satisfies a certain bound. Since $\boldsymbol{e}$ is fixed once a generating vector $\boldsymbol{g}^\ast$ is chosen, we consider now the set of weights
\begin{equation*}
\Gamma = \left\{(\gamma_u)_{\emptyset \neq u \subseteq \mathcal{S}}: e^2(\mathcal{H}_{\boldsymbol{\gamma}}, P_N(\boldsymbol{g}^\ast)) \le \varepsilon e^2(\mathcal{H}_{\boldsymbol{\gamma}}, P_0) \right\},
\end{equation*}
where $\varepsilon > 0$ is a real number, $P_0 = \emptyset$ and
\begin{equation*}
e^2(\mathcal{H}_{\boldsymbol{\gamma}}, P_0) = \inf_{\satop{f \in \mathcal{H}_{\boldsymbol{\gamma}}}{\|f\|_{\mathcal{H}_{\boldsymbol{\gamma}} \le 1}}} \left|\int_{[0,1]^s} f(\boldsymbol{x}) \,\mathrm{d} \boldsymbol{x} \right|
\end{equation*}
is the initial error \cite{SW98}. For our space we have $e^2(\mathcal{H}_{\boldsymbol{\gamma}}, P_0) = 1$.

The square worst-case error $e^2(\mathcal{H}_{\boldsymbol{\gamma}}, P_N(\boldsymbol{g}^\ast)) = \boldsymbol{e} \cdot \boldsymbol{\gamma}$ is a linear function of $\gamma_u$. Thus $\Gamma$ is a simplex in $\mathbb{R}^{2^s-1}$ given by
\begin{equation*}
\Gamma = \Gamma(\boldsymbol{e}, \varepsilon),
\end{equation*}
which has vertices $(0)_{\emptyset \neq u \subseteq \mathcal{S}}$ and $$\left(\boldsymbol{0}_{\mathcal{S}\setminus u},  e_u^{-1} \varepsilon \right) \quad \mbox{for } \emptyset \neq u \subseteq \mathcal{S}.$$

\begin{figure}[ht]
\begin{center} 
\includegraphics[scale= 0.8]{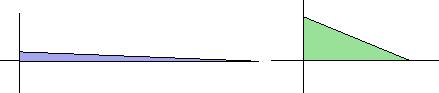}
\caption{\label{fig2} The set of weights $\Gamma$. In the first instance, one value of $e_u$ (corresponding to the $y$-axis in the picture) is large, therefore the algorithm is sensitive to changes with respect to the projection $u$ (the corresponding weight $\gamma_u$ has to be small). A small change in $\gamma_u$ could significantly increase the error. In the second case the algorithm is more robust since bigger changes in the second coordinate are allowed.}
\end{center}
\end{figure}

We consider now the set of weights for which the cbc$r$c algorithm yields bounds. We have the following result.
\begin{theorem}\label{thm2}
Let $\boldsymbol{g}^\ast$ be constructed by the cbc$r$c algorithm using the weights $\boldsymbol{\gamma}^{(1)},\ldots, \boldsymbol{\gamma}^{(r)}$. Let $\boldsymbol{\gamma} = \sum_{w=1}^r \lambda_w \boldsymbol{\gamma}^{(w)}$ for some $\lambda_1,\ldots, \lambda_r \ge 0$. Then it follows that
\begin{equation*}
e^2(\mathcal{H}_{\boldsymbol{\gamma}}, P_N(\boldsymbol{g}^\ast)) \le \sum_{w=1}^{r} \lambda_w X_{\boldsymbol{\gamma}^{(w)}, c_w}.
\end{equation*}
\end{theorem}

\begin{proof}
By the cbc$r$c algorithm we have $e^2(\mathcal{H}_{\boldsymbol{\gamma}^{(w)}}, P_N(\boldsymbol{g}^\ast)) = \boldsymbol{e} \cdot \boldsymbol{\gamma}^{(w)} \le X_{\boldsymbol{\gamma}^{(w)},c_w}$ for $1 \le w \le r$. Thus we have
\begin{equation*}
e^2(\mathcal{H}_{\boldsymbol{\gamma}}, P_N(\boldsymbol{g}^\ast)) = \boldsymbol{e} \cdot \boldsymbol{\gamma} = \sum_{w=1}^r \lambda_w \boldsymbol{e} \cdot \boldsymbol{\gamma}^{(w)} = \sum_{w=1}^r \lambda_w e^2(\mathcal{H}_{\boldsymbol{\gamma}^{(w)}}, P_N(\boldsymbol{g}^\ast))  \le \sum_{w=1}^r \lambda_w X_{\boldsymbol{\gamma}^{(w)},c_w},
\end{equation*}
which shows the result.
\end{proof}
The above theorem shows that the cbc$r$c algorithm yields lattice rules which simultaneously satisfy bounds for weights $\boldsymbol{\gamma} \in \mathbb{R}^{2^s-1}$ taken from a subspace of $\mathbb{R}^{2^s-1}$ spanned by $\boldsymbol{\gamma}^{(1)},\ldots, \boldsymbol{\gamma}^{(r)} \in \mathbb{R}^{2^s-1}$.

The theorem above can be understood geometrically in the following way. Note that $\boldsymbol{e} \in \Gamma(\boldsymbol{\gamma}^{(w)},X_{\boldsymbol{\gamma}^{(w)}, c_w})$ if and only if $\boldsymbol{\gamma}^{(w)} \in \Gamma(\boldsymbol{e},X_{\boldsymbol{\gamma}^{(w)},c_w})$. We define the new weights $\widehat{\boldsymbol{\gamma}}^{(w)} = \varepsilon X_{\boldsymbol{\gamma}^{(w)},c_w}^{-1} \boldsymbol{\gamma}^{(w)}$, $1 \le w \le r$ for some real number $\varepsilon > 0$. Then $X_{\widehat{\boldsymbol{\gamma}}^{(w)},c_w} = \varepsilon$ for all $1 \le w \le r$. Note that criterion in the cbc$r$c algorithm does not change by this normalization since the cbc$r$c algorithm yields exactly the same generating vector $\boldsymbol{g}^\ast$ using the weights $\boldsymbol{\gamma}^{(1)},\ldots, \boldsymbol{\gamma}^{(r)}$ as it does for using the weights $\widehat{\boldsymbol{\gamma}}^{(1)}, \ldots, \widehat{\boldsymbol{\gamma}}^{(r)}$. Thus
\begin{equation*}
\boldsymbol{e} \in \Gamma(\boldsymbol{\gamma}^{(w)},X_{\boldsymbol{\gamma}^{(w)}, c_w}) \Leftrightarrow \boldsymbol{\gamma}^{(w)} \in \Gamma(\boldsymbol{e},X_{\boldsymbol{\gamma}^{(w)},c_w}) \Leftrightarrow \widehat{\boldsymbol{\gamma}}^{(w)} \in \Gamma(\boldsymbol{e}, \varepsilon).
\end{equation*}
Now \eqref{eq_simplex} is equivalent to
\begin{equation*}
\boldsymbol{e} \in \bigcap_{1 \le w \le r} \Gamma(\boldsymbol{\gamma}^{(w)}, X_{\boldsymbol{\gamma}^{(w)},c_w}) \Leftrightarrow \widehat{\boldsymbol{\gamma}}^{(w)} \in \Gamma(\boldsymbol{e}, \varepsilon) \quad \mbox{for all } 1 \le w \le r.
\end{equation*}
Therefore, the cbc$r$c algorithm ensures that $\boldsymbol{g}^\ast$ is chosen such that $\widehat{\boldsymbol{\gamma}}^{(1)},\ldots, \widehat{\boldsymbol{\gamma}}^{(r)}$ all lie in the simplex $\Gamma(\boldsymbol{e}, \varepsilon )$, i.e. $$\widehat{\boldsymbol{\gamma}}^{(1)},\ldots, \widehat{\boldsymbol{\gamma}}^{(r)} \in \Gamma(\boldsymbol{e}, \varepsilon ).$$ In fact, the cbc$r$c algorithm finds, component-by-component, the smallest simplex $\Gamma(\boldsymbol{e},\varepsilon)$ which contains $\widehat{\boldsymbol{\gamma}}^{(1)},\ldots, \widehat{\boldsymbol{\gamma}}^{(r)}$.

If one chooses $r=2^s-1$ in the cbc$r$c algorithm, then the weights $\widehat{\boldsymbol{\gamma}}^{(1)},\ldots, \widehat{\boldsymbol{\gamma}}^{(2^s-1)} \in \Gamma(\boldsymbol{e}, \varepsilon)$ themselves are the vertices of a $2^s-1$-dimensional convex polytope which is contained in the simplex $\Gamma(\boldsymbol{e}, \varepsilon)$ of the same dimension. However, if $1 \le r < 2^s-1$, then the convex polytope spanned by $\widehat{\boldsymbol{\gamma}}^{(1)},\ldots, \widehat{\boldsymbol{\gamma}}^{(r)}$ is degenerate since it lies in a $r$-dimensional subspace. Thus using only $r < 2^s-1$ weights $\boldsymbol{\gamma}^{(1)}, \ldots, \boldsymbol{\gamma}^{(r)}$ does not fully control the shape of the simplex $\Gamma(\boldsymbol{e}, \varepsilon)$. For the classical cbc algorithm only one vector of weights $\boldsymbol{\gamma}$ is used. The cbc construction then only ensures that $\widehat{\boldsymbol{\gamma}} \in \Gamma(\boldsymbol{e}, \varepsilon )$. In the numerical examples below we show that it is possible for the classical cbc construction to choose generating vectors which are not suitable for many other choices of weights. By adding additional constraints, the cbc$r$c algorithm can prevent such bad choices.

The bound in Theorem~\ref{thm2} applies for all weights which lie in the linear subspace of $\mathbb{R}^{2^s-1}$ spanned by vectors $\boldsymbol{\gamma}^{(1)},\ldots, \boldsymbol{\gamma}^{(r)}$. In particular, if one uses the cbc$(2^s-1)$c algorithm, then one can obtain a bound for any choice of weights. However, in higher dimensions $s$ this is currently problematic for two reasons: the computational cost is exponential in the dimension; the second problem is that one would have to choose $c_1,\ldots, c_{2^s-1}$ such that $c_1^{-1} + \cdots + c_{2^s-1}^{-1} = 1$. For instance, the choice $c_j = 2^{s}-1$, $1 \le j \le s$, yields a factor in the upper bound which grows exponentially with the dimension. For lower dimensions this is feasible though and hence can be useful in applications with low truncation dimension.

Consider now $\boldsymbol{\gamma} = \sum_{w=1}^r \lambda_w \widehat{\boldsymbol{\gamma}}^{(w)}$, where $\lambda_1,\ldots, \lambda_r \ge 0$ and $\lambda_1 + \cdots + \lambda_r \le 1$, i.e., $\boldsymbol{\gamma}$ lies in the convex polytope with vertices $\boldsymbol{0}, \widehat{\boldsymbol{\gamma}}^{(1)}, \ldots, \widehat{\boldsymbol{\gamma}}^{(r)}$.Then, by Theorem~\ref{thm2}, we have $$e^2(\mathcal{H}_{\boldsymbol{\gamma}}, P_N(\boldsymbol{g}^\ast)) \le \varepsilon.$$ We summarize the results in the following corollary.

\begin{corollary}
Let $1 \le r \le 2^s-1$. Let $c_1,\ldots, c_r \ge 1$ be given such that $c_1^{-1} + \cdots + c_{r}^{-1} = 1$. Let $\varepsilon > 0$ be a real number. Let $\widehat{\boldsymbol{\gamma}}^{(1)},\ldots, \widehat{\boldsymbol{\gamma}}^{(r)}$ be weights which are normalized such that $X_{\widehat{\boldsymbol{\gamma}}^{(w)}, c_w} = \varepsilon$ for all $1 \le w \le r$. Let $\boldsymbol{\gamma} = \sum_{w=1}^r \lambda_w \widehat{\boldsymbol{\gamma}}^{(w)}$ where $\lambda_1,\ldots, \lambda_r \ge 0$ and $\lambda_1 + \cdots + \lambda_r \le 1$. Let $\boldsymbol{g}^\ast$ be constructed by the cbc$r$c algorithm based on the weights $\widehat{\boldsymbol{\gamma}}^{(1)},\ldots, \widehat{\boldsymbol{\gamma}}^{(r)}$. Then
\begin{equation*}
\widehat{\boldsymbol{\gamma}}^{(1)},\ldots, \widehat{\boldsymbol{\gamma}}^{(r)} \in \Gamma(\boldsymbol{e}, \varepsilon)
\end{equation*}
and
\begin{equation*}
e^2(\mathcal{H}_{\boldsymbol{\gamma}}, P_N(\boldsymbol{g}^\ast)) \le \varepsilon.
\end{equation*}
\end{corollary}

\section{Numerical results}\label{sec_num}

To illustrate the ideas in the paper we chose some instructive examples. We tested the algorithm with $r=2$. The computation time for the cbc$2$c algorithm is between $2$ and $3$ times the computation time of the cbc$1$c algorithm. This was observed for a variety of choices for $c_1$ and $c_2$ and values $n$.

Further we tested the component-by-component algorithm with fast decaying weights. We used  $s=100$, product weights $\gamma_{u}^{(w)} = \prod_{j \in u} \widehat{\gamma}_j^{(w)}$ with $\widehat{\gamma}^{(1)}_j= 1$, $\widehat{\gamma}_j^{(2)} = 10^{-j}$ and $\widehat{\gamma}_j^{(3)} = j^{-1}$. Further we chose $c_1=c_2=2$. The cbc$1$c algorithm using the weights $\boldsymbol{\gamma}^{(1)}$ returns the same components $g_j^{\ast} = g_{j'}^\ast$ for $j, j' \ge 15$. This choice of generating vector would be bad for slow decaying weights. On the other hand, the cbc$2$c does not return any repeated components, which is prevented by the second constraint. The results are presented in Table~\ref{table1}. It shows that the cbc$2$c algorithm yields approximately the same results as the cbc$1$c algorithm constructed for the right weight, but can do significantly better if the lattice rule constructed by the cbc$1$c algorithm is used for different weights, as can be seen in Table~\ref{table1}. The results in Table~\ref{table2} are for a different choice of weights and are similar.

The choice of weights $\widehat{\gamma}_j^{(1)}$ and $\widehat{\gamma}_j^{(2)}$ in Table~\ref{table1} and Table~\ref{table2} are quite different from each other. Numerical tests for examples where the weights are more similar than in the examples shown, for instance $\widehat{\gamma}_j^{(1)} = j^{-1}$ and $\widehat{\gamma}_j^{(2)} = j^{-4}$ (or even $\widehat{\gamma}_j^{(2)} = 2^{-j}$), yield numerical results which are quite similar, indicating that there are not many (bad) outliers. Table~\ref{table3} and \ref{table4} show numerical results with randomly chosen weights, again showing that the lattice rules constructed by the cbc and cbc$2$c algorithms perform well except in the case where the cbc construction is based on fast decaying weights.

\begin{table}
\begin{tabular}{r||r|r|r|r|r}
  N & 251 & 509  & 1019 & 2039 & 4079  \\ \hline
  cbc$2$c: $e(\boldsymbol{\gamma}^{(1)})$ & 1.4044e+02 &9.8623e+01 & 6.9702e+01&4.9275e+01 & 3.4838e+01 \\
  cbc($\boldsymbol{\gamma}^{(1)}$): $e(\boldsymbol{\gamma}^{(1)})$ & 1.4044e+02&9.8623e+01 & 6.9702e+01& 4.9274e+01& 3.4838e+01 \\
  cbc($\boldsymbol{\gamma}^{(2)}$): $e(\boldsymbol{\gamma}^{(1)})$ & 2.4075e+02 & 2.1790e+02 & 1.9762e+02 & 1.9481e+02 &  1.8137e+02 \\ \hline
  cbc$2$c: $e(\boldsymbol{\gamma}^{(2)})$ & 5.4897e-04 & 2.7128e-04 & 1.3568e-04 & 6.7927e-05 & 3.3965e-05 \\
  cbc($\boldsymbol{\gamma}^{(1)}$): $e(\boldsymbol{\gamma}^{(2)})$ & 5.4897e-04 & 2.7128e-04 & 1.3568e-04 &6.7930e-05  & 3.3966e-05  \\
  cbc($\boldsymbol{\gamma}^{(2)}$): $e(\boldsymbol{\gamma}^{(2)})$ & 5.4882e-04 & 2.7113e-04 & 1.3558e-04 & 6.7892e-05 &  3.3954e-05  \\ \hline
  cbc$2$c: $e(\boldsymbol{\gamma}^{(3)})$ & 3.1971e-02 & 1.9872e-02 &1.2057e-02  & 7.7449e-03 &4.9349e-03  \\
  cbc($\boldsymbol{\gamma}^{(1)}$): $e(\boldsymbol{\gamma}^{(3)})$ & 3.1653e-02 & 2.0250e-02 & 1.3100e-02 & 8.5062e-03 & 4.9972e-03  \\
  cbc($\boldsymbol{\gamma}^{(2)}$): $e(\boldsymbol{\gamma}^{(3)})$ & 1.5597e-01 & 1.4616e-01 & 1.3816e-01 & 1.3813e-01 &  1.3130e-01 \\
  \hline\hline
  N & 8161 & 16319 & 32633 & 65267 & 130531  \\ \hline
  cbc$2$c: $e(\boldsymbol{\gamma}^{(1)})$  & 2.4629e+01 & 1.7417e+01 & 1.2316e+01 & 8.7088e+00 & 6.1579e+00 \\
  cbc($\boldsymbol{\gamma}^{(1)}$): $e(\boldsymbol{\gamma}^{(1)})$  & 2.4629e+01 & 1.7417e+01 & 1.2316e+01 & 8.7087e+00 & 6.1579e+00 \\
  cbc($\boldsymbol{\gamma}^{(2)}$): $e(\boldsymbol{\gamma}^{(1)})$ & 1.6777e+02 & 1.6737e+02 & 1.5575e+02 & 1.5567e+02 & 1.5563e+02 \\ \hline
  cbc$2$c: $e(\boldsymbol{\gamma}^{(2)})$  & 1.7023e-05 & 8.5236e-06 & 4.2695e-06 & 2.1370e-06 & 1.0753e-06 \\
  cbc($\boldsymbol{\gamma}^{(1)}$): $e(\boldsymbol{\gamma}^{(2)})$ & 1.7023e-05 & 8.5236e-06 & 4.2690e-06 & 2.1370e-06 & 1.0721e-06  \\
  cbc($\boldsymbol{\gamma}^{(2)}$): $e(\boldsymbol{\gamma}^{(2)})$  & 1.7006e-05 & 8.5111e-06 & 4.2631e-06 & 2.1351e-06 & 1.0683e-06 \\ \hline
  cbc$2$c: $e(\boldsymbol{\gamma}^{(3)})$ & 3.0911e-03 & 2.0308e-03 & 1.2551e-03 & 7.9994e-04 & 5.2220e-04 \\
  cbc($\boldsymbol{\gamma}^{(1)}$): $e(\boldsymbol{\gamma}^{(3)})$ & 3.2965e-03 & 2.1254e-03 & 1.3246e-03 & 8.5575e-04 & 5.8758e-04  \\
  cbc($\boldsymbol{\gamma}^{(2)}$): $e(\boldsymbol{\gamma}^{(3)})$ & 1.2522e-01 & 1.2521e-01 & 1.1979e-01 & 1.1979e-01 &  1.1979e-01 \\
  \hline\hline
  N & 261061 & 522127 & 1044257 &  2088511 & 4177051  \\ \hline
  cbc$2$c: $e(\boldsymbol{\gamma}^{(1)})$  & 4.3542e+00 & 3.0787e+00 & 2.1769e+00 & 1.5392e+00 &  1.0883e+00 \\
  cbc($\boldsymbol{\gamma}^{(1)}$): $e(\boldsymbol{\gamma}^{(1)})$  & 4.3542e+00 & 3.0787e+00 & 2.1769e+00 & 1.5392e+00 & 1.0883e+00 \\
  cbc($\boldsymbol{\gamma}^{(2)}$): $e(\boldsymbol{\gamma}^{(1)})$ & 1.3495e+02 & 1.1735e+02 & 1.2566e+02 &  1.1702e+02 &  1.1701e+02 \\ \hline
  cbc$2$c: $e(\boldsymbol{\gamma}^{(2)})$  & 5.3706e-07 & 2.7758e-07 & 1.3899e-07 & 6.3220e-08 & 0.0000e+00  \\
  cbc($\boldsymbol{\gamma}^{(1)}$): $e(\boldsymbol{\gamma}^{(2)})$ & 5.3686e-07 & 2.5938e-07 & 1.3899e-07 & 0.0000e+00 &  7.3000e-08 \\
  cbc($\boldsymbol{\gamma}^{(2)}$): $e(\boldsymbol{\gamma}^{(2)})$  & 5.3541e-07 & 2.5981e-07 & 1.3411e-07 & 1.1151e-07 & 1.6255e-07 \\ \hline
  cbc$2$c: $e(\boldsymbol{\gamma}^{(3)})$  & 3.2756e-04 & 2.1752e-04 & 1.4107e-04 & 8.6973e-05 & 5.7966e-05 \\
  cbc($\boldsymbol{\gamma}^{(1)}$): $e(\boldsymbol{\gamma}^{(3)})$ & 4.2678e-04 & 2.1369e-04 & 1.3647e-04 & 9.4174e-05 & 5.7208e-05  \\
  cbc($\boldsymbol{\gamma}^{(2)}$): $e(\boldsymbol{\gamma}^{(3)})$ & 1.1070e-01 & 1.0250e-01 & 1.0620e-01 & 1.0234e-01 &  1.0234e-01 
\end{tabular}
\caption{This table shows the square worst-case errors using the cbc$2$c construction based on the weights $\boldsymbol{\gamma}^{(1)}$ and $\boldsymbol{\gamma}^{(2)}$, the cbc construction based on the weights $\boldsymbol{\gamma}^{(1)}$ and the cbc construction based on the weights $\boldsymbol{\gamma}^{(2)}$. Here, $e(\boldsymbol{\gamma}^{(w)})$ stands for the worst-case error $e(\mathcal{H}_{\boldsymbol{\gamma}^{(w)}}, P_N(\boldsymbol{g}^\ast))$. We choose $s = 100$, product weights with $\widehat{\gamma}^{(1)}_j = 1$, $\widehat{\gamma}_j^{(2)}=10^{-j}$, $\widehat{\gamma}_j^{(3)} = j^{-1}$, $c_1=c_2=2$;}\label{table1}
\end{table}

\begin{table}
\begin{tabular}{r||r|r|r|r|r}
  N & 251 & 509  & 1019 & 2039 & 4079  \\ \hline
  cbc$2$c: $e(\boldsymbol{\gamma}^{(1)})$ & 3.9170e-03 & 2.0603e-03 & 1.1145e-03 & 6.0564e-04 &  3.2079e-04 \\
  cbc($\boldsymbol{\gamma}^{(1)}$): $e(\boldsymbol{\gamma}^{(1)})$ & 3.9148e-03& 2.0587e-03 & 1.1142e-03& 6.0547e-04& 3.2073e-04 \\
  cbc($\boldsymbol{\gamma}^{(2)}$): $e(\boldsymbol{\gamma}^{(1)})$ & 9.5094e-03 & 7.6012e-03 & 6.4146e-03 &6.3616e-03  &  5.5514e-03 \\ \hline
  cbc$2$c: $e(\boldsymbol{\gamma}^{(2)})$ & 5.4893e-04 & 2.7113e-04 & 1.3558e-04 & 6.7914e-05 & 3.3957e-05 \\
  cbc($\boldsymbol{\gamma}^{(1)}$): $e(\boldsymbol{\gamma}^{(2)})$ & 5.4893e-04 & 2.7113e-04 & 1.3558e-04 & 6.7914e-05 & 3.3957e-05  \\
  cbc($\boldsymbol{\gamma}^{(2)}$): $e(\boldsymbol{\gamma}^{(2)})$ & 5.4882e-04 & 2.7113e-04 & 1.3558e-04 & 6.7892e-05 &  3.3954e-05 \\ \hline
  cbc$2$c: $e(\boldsymbol{\gamma}^{(3)})$ & 5.3522e-03 & 3.2454e-03 & 2.1131e-03 & 9.9811e-04 & 5.6878e-04 \\
  cbc($\boldsymbol{\gamma}^{(1)}$): $e(\boldsymbol{\gamma}^{(3)})$ & 5.9244e-03 & 2.9483e-03 & 1.9003e-03 & 1.0215e-03 & 5.9843e-04  \\
  cbc($\boldsymbol{\gamma}^{(2)}$): $e(\boldsymbol{\gamma}^{(3)})$ & 3.0073e-02 & 2.9957e-02 & 2.9922e-02 & 2.9914e-02 &  2.9906e-02 \\
  \hline\hline
  N & 8161 & 16319 & 32633 & 65267 & 130531  \\ \hline
  cbc$2$c: $e(\boldsymbol{\gamma}^{(1)})$  & 1.7664e-04 & 9.5576e-05 & 5.2562e-05 & 2.8616e-05 & 1.5674e-05 \\
  cbc($\boldsymbol{\gamma}^{(1)}$): $e(\boldsymbol{\gamma}^{(1)})$  & 1.7664e-04 & 9.5576e-05 & 5.2559e-05 & 2.8616e-05 & 1.5674e-05\\
  cbc($\boldsymbol{\gamma}^{(2)}$): $e(\boldsymbol{\gamma}^{(1)})$ & 4.9127e-03 & 4.9093e-03 & 4.3999e-03 & 4.3994e-03 & 4.3993e-03 \\ \hline
  cbc$2$c: $e(\boldsymbol{\gamma}^{(2)})$  & 1.7006e-05 & 8.5147e-06 & 4.2646e-06 & 2.1341e-06 & 1.0650e-06 \\
  cbc($\boldsymbol{\gamma}^{(1)}$): $e(\boldsymbol{\gamma}^{(2)})$ & 1.7006e-05 & 8.5147e-06 & 4.2646e-06 & 2.1341e-06 & 1.0650e-06  \\
  cbc($\boldsymbol{\gamma}^{(2)}$): $e(\boldsymbol{\gamma}^{(2)})$  & 1.7006e-05 & 8.5111e-06 &4.2631e-06  & 2.1351e-06 & 1.0683e-06 \\ \hline
  cbc$2$c: $e(\boldsymbol{\gamma}^{(3)})$ & 3.5384e-04 & 1.9131e-04 & 1.1179e-04 & 6.7310e-05 & 4.8831e-05 \\
  cbc($\boldsymbol{\gamma}^{(1)}$): $e(\boldsymbol{\gamma}^{(3)})$ & 3.5384e-04 & 1.9131e-04 & 1.0812e-04 & 6.7310e-05 & 4.2383e-05  \\
  cbc($\boldsymbol{\gamma}^{(2)}$): $e(\boldsymbol{\gamma}^{(3)})$ & 2.9898e-02 & 2.9898e-02 & 2.9891e-02 & 2.9891e-02 &  2.9891e-02 \\
  \hline\hline
  N & 261061 & 522127 & 1044257 &  2088511 & 4177051  \\ \hline
  cbc$2$c: $e(\boldsymbol{\gamma}^{(1)})$  & 8.6033e-06 & 4.7039e-06 & 2.6006e-06 & 1.4241e-06 & 8.0218e-07  \\
  cbc($\boldsymbol{\gamma}^{(1)}$): $e(\boldsymbol{\gamma}^{(1)})$  & 8.6019e-06 & 4.7038e-06 & 2.5996e-06 & 1.4242e-06 & 8.0024e-07 \\
  cbc($\boldsymbol{\gamma}^{(2)}$): $e(\boldsymbol{\gamma}^{(1)})$ & 3.6719e-03 & 3.0833e-03 & 3.3182e-03 & 3.0560e-03 &  3.0560e-03 \\ \hline
  cbc$2$c: $e(\boldsymbol{\gamma}^{(2)})$  & 5.3458e-07 & 2.7028e-07 & 1.1921e-07 & 0.0000e+00 & 0.0000e+00  \\
  cbc($\boldsymbol{\gamma}^{(1)}$): $e(\boldsymbol{\gamma}^{(2)})$ & 5.3458e-07 & 2.7028e-07 & 1.1921e-07 & 0.0000e+00 & 0.0000e+00  \\
  cbc($\boldsymbol{\gamma}^{(2)}$): $e(\boldsymbol{\gamma}^{(2)})$  & 5.3541e-07 & 2.5981e-07 & 1.3411e-07 & 1.1151e-07 & 1.6255e-07 \\ \hline
  cbc$2$c: $e(\boldsymbol{\gamma}^{(3)})$  & 2.0703e-05 & 1.2833e-05 & 7.2400e-06 & 5.4021e-06 & 2.7314e-06 \\
  cbc($\boldsymbol{\gamma}^{(1)}$): $e(\boldsymbol{\gamma}^{(3)})$ & 2.3378e-05 & 1.2218e-05 & 7.6749e-06 & 3.9504e-06 &  2.0520e-06 \\
  cbc($\boldsymbol{\gamma}^{(2)}$): $e(\boldsymbol{\gamma}^{(3)})$ & 2.9876e-02 & 2.9860e-02 & 2.9868e-02 & 2.9860e-02 &  2.9860e-02 
\end{tabular}
\caption{This table shows the square worst-case errors using the cbc$2$c construction based on the weights $\boldsymbol{\gamma}^{(1)}$ and $\boldsymbol{\gamma}^{(2)}$, the cbc construction based on the weights $\boldsymbol{\gamma}^{(1)}$ and the cbc construction based on the weights $\boldsymbol{\gamma}^{(2)}$. Here, $e(\boldsymbol{\gamma}^{(w)})$ stands for the worst-case error $e(\mathcal{H}_{\boldsymbol{\gamma}^{(w)}}, P_N(\boldsymbol{g}^\ast))$. We choose $s = 100$, product weights with $\widehat{\gamma}^{(1)}_j = 1$, $\widehat{\gamma}_j^{(2)}  = j^{-2}$, $\widehat{\gamma}_j^{(3)} = (s-j)^{-2}$, $c_1=c_2=2$;}\label{table2}
\end{table}

\begin{table}
\begin{tabular}{r||r|r|r|r|r}
  N & 251 & 509  & 1019 & 2039 & 4079  \\ \hline
  cbc$2$c: $e(\boldsymbol{\gamma}^{(1)})$ & 3.0814e-02& 1.8765e-02 & 1.1539e-02 &7.0582e-03 & 4.4014e-03 \\
  cbc($\boldsymbol{\gamma}^{(1)}$): $e(\boldsymbol{\gamma}^{(1)})$ & 2.9982e-02& 1.8448e-02& 1.1482e-02 &7.0420e-03 & 4.3849e-03 \\
  cbc($\boldsymbol{\gamma}^{(2)}$): $e(\boldsymbol{\gamma}^{(1)})$ & 8.5356e-02 & 7.9743e-02 & 7.7126e-02 & 6.7998e-02 & 6.7529e-02 \\ \hline
  cbc$2$c: $e(\boldsymbol{\gamma}^{(2)})$ & 2.4617e-03 & 1.2385e-03& 6.6254e-04 & 3.5517e-04 & 1.8239e-04 \\
  cbc($\boldsymbol{\gamma}^{(1)}$): $e(\boldsymbol{\gamma}^{(2)})$ & 2.4617e-03 & 1.2385e-03 & 6.6254e-04 & 3.5517e-04 &  1.8239e-04 \\
  cbc($\boldsymbol{\gamma}^{(2)}$): $e(\boldsymbol{\gamma}^{(2)})$ & 2.4416e-03 & 1.2423e-03 & 6.5820e-04 & 3.4793e-04 &  1.7957e-04  \\ \hline
  cbc$2$c: $e(\boldsymbol{\gamma}^{(3)})$ & 2.5477e+00 & 1.7828e+00 & 1.2581e+00 & 8.8750e-01 & 6.2621e-01 \\
  cbc($\boldsymbol{\gamma}^{(1)}$): $e(\boldsymbol{\gamma}^{(3)})$ & 2.5413e+00 & 1.7826e+00 & 1.2576e+00 & 8.8690e-01 & 6.2570e-01  \\
  cbc($\boldsymbol{\gamma}^{(2)}$): $e(\boldsymbol{\gamma}^{(3)})$ & 4.0922e+00 & 4.7043e+00 & 3.7287e+00 & 3.1589e+00 &  3.2877e+00 \\
  \hline\hline
  N & 8161 & 16319 & 32633 & 65267 & 130531  \\ \hline
  cbc$2$c: $e(\boldsymbol{\gamma}^{(1)})$  & 2.7229e-03& 1.6958e-03 & 1.0601e-03 & 6.6402e-04 & 4.1363e-04 \\
  cbc($\boldsymbol{\gamma}^{(1)}$): $e(\boldsymbol{\gamma}^{(1)})$  & 2.7208e-03 & 1.6957e-03 & 1.0587e-03 & 6.6282e-04 & 4.1368e-04 \\
  cbc($\boldsymbol{\gamma}^{(2)}$): $e(\boldsymbol{\gamma}^{(1)})$  & 6.1343e-02 & 5.5530e-02 & 5.3283e-02 & 5.2673e-02 & 5.0934e-02 \\ \hline
  cbc$2$c: $e(\boldsymbol{\gamma}^{(2)})$  & 9.5776e-05 & 4.9506e-05 & 2.6222e-05 & 1.3943e-05 & 7.3677e-06 \\
  cbc($\boldsymbol{\gamma}^{(1)}$): $e(\boldsymbol{\gamma}^{(2)})$ & 9.5776e-05 & 4.9506e-05 & 2.6222e-05 & 1.3943e-05 &  7.3677e-06 \\
  cbc($\boldsymbol{\gamma}^{(2)}$): $e(\boldsymbol{\gamma}^{(2)})$ & 9.4743e-05 & 4.9263e-05 & 2.5759e-05 & 1.3567e-05 & 7.2127e-06  \\ \hline
  cbc$2$c: $e(\boldsymbol{\gamma}^{(3)})$ & 4.4091e-01 & 3.1093e-01 & 2.1929e-01 & 1.5442e-01 & 1.0883e-01 \\
  cbc($\boldsymbol{\gamma}^{(1)}$): $e(\boldsymbol{\gamma}^{(3)})$ & 4.4099e-01 & 3.1075e-01 & 2.1886e-01 & 1.5431e-01 & 1.0875e-01  \\
  cbc($\boldsymbol{\gamma}^{(2)}$): $e(\boldsymbol{\gamma}^{(3)})$ & 3.1957e+00 & 2.6282e+00 & 2.7940e+00 & 2.5026e+00 & 2.3333e+00  \\
  \hline\hline
  N & 261061 & 522127 & 1044257 &  2088511 & 4177051  \\ \hline
  cbc$2$c: $e(\boldsymbol{\gamma}^{(1)})$  & 2.5887e-04 & 1.6206e-04 & 1.0112e-04 & 6.3253e-05 & 3.9582e-05  \\
  cbc($\boldsymbol{\gamma}^{(1)}$): $e(\boldsymbol{\gamma}^{(1)})$  & 2.5874e-04 & 1.6160e-04 & 1.0101e-04 & 6.3241e-05 & 3.9495e-05 \\
  cbc($\boldsymbol{\gamma}^{(2)}$): $e(\boldsymbol{\gamma}^{(1)})$  & 4.6470e-02 & 4.7789e-02 & 4.6729e-02 & 4.1838e-02 & 4.5181e-02  \\ \hline
  cbc$2$c: $e(\boldsymbol{\gamma}^{(2)})$  & 3.8690e-06 & 2.0481e-06 & 1.0884e-06 & 5.8400e-07 &  2.8115e-07 \\
  cbc($\boldsymbol{\gamma}^{(1)}$): $e(\boldsymbol{\gamma}^{(2)})$ & 3.8690e-06 & 2.0481e-06 & 1.0884e-06 & 5.8419e-07 & 2.6781e-07  \\
  cbc($\boldsymbol{\gamma}^{(2)}$): $e(\boldsymbol{\gamma}^{(2)})$ & 3.7567e-06 & 1.9802e-06 & 1.0503e-06 & 5.3995e-07 &  3.0465e-07\\ \hline
  cbc$2$c: $e(\boldsymbol{\gamma}^{(3)})$  & 7.6712e-02 & 5.3901e-02 & 3.7986e-02 & 2.6700e-02 & 1.8762e-02 \\
  cbc($\boldsymbol{\gamma}^{(1)}$): $e(\boldsymbol{\gamma}^{(3)})$ & 7.6574e-02 & 5.4043e-02 & 3.7979e-02 & 2.6664e-02 & 1.8764e-02  \\
  cbc($\boldsymbol{\gamma}^{(2)}$): $e(\boldsymbol{\gamma}^{(3)})$ & 2.2229e+00 & 2.1986e+00 & 2.3036e+00 & 2.0369e+00 & 2.1620e+00  
\end{tabular}
\caption{This table shows the square worst-case errors using the cbc$2$c construction based on the weights $\boldsymbol{\gamma}^{(1)}$ and $\boldsymbol{\gamma}^{(2)}$, the cbc construction based on the weights $\boldsymbol{\gamma}^{(1)}$ and the cbc construction based on the weights $\boldsymbol{\gamma}^{(2)}$. Here, $e(\boldsymbol{\gamma}^{(w)})$ stands for the worst-case error $e(\mathcal{H}_{\boldsymbol{\gamma}^{(w)}}, P_N(\boldsymbol{g}^\ast))$. We choose $s = 100$, product weights with $\widehat{\gamma}^{(1)}_j = j^{-1}$, $\widehat{\gamma}_j^{(2)}= 2^{-j}$, $\widehat{\gamma}_j^{(3)}$ is chosen randomly in $[0,1]$, $c_1=c_2=2$;}\label{table3}
\end{table}

\begin{table}
\begin{tabular}{r||r|r|r|r|r}
  N & 251 & 509  & 1019 & 2039 & 4079  \\ \hline
  cbc$2$c: $e(\boldsymbol{\gamma}^{(1)})$ & 2.4416e-03 & 1.2423e-03& 6.5820e-04& 3.4797e-04& 1.7957e-04 \\
  cbc($\boldsymbol{\gamma}^{(1)}$): $e(\boldsymbol{\gamma}^{(1)})$ & 2.4416e-03&1.2423e-03 &6.5820e-04 & 3.4793e-04& 1.7957e-04 \\
  cbc($\boldsymbol{\gamma}^{(2)}$): $e(\boldsymbol{\gamma}^{(1)})$ & 2.4832e-03 & 1.2705e-03 &  6.6492e-04& 3.6102e-04 &  1.8731e-04 \\ \hline
  cbc$2$c: $e(\boldsymbol{\gamma}^{(2)})$ & 3.4093e+00 & 2.3917e+00 & 1.6885e+00 & 1.1923e+00 & 8.4217e-01 \\
  cbc($\boldsymbol{\gamma}^{(1)}$): $e(\boldsymbol{\gamma}^{(2)})$ & 5.3545e+00 & 6.1955e+00 & 4.8321e+00 & 3.9970e+00 &  4.2535e+00 \\
  cbc($\boldsymbol{\gamma}^{(2)}$): $e(\boldsymbol{\gamma}^{(2)})$ & 3.4052e+00 & 2.3878e+00 &1.6849e+00  & 1.1889e+00 & 8.3874e-01  \\ \hline
  cbc$2$c: $e(\boldsymbol{\gamma}^{(3)})$ & 1.4044e+02 & 9.8623e+01 & 6.9703e+01 & 4.9275e+01 & 3.4838e+01 \\
  cbc($\boldsymbol{\gamma}^{(1)}$): $e(\boldsymbol{\gamma}^{(3)})$ & 1.4761e+02 & 1.2953e+02 & 8.4534e+01 & 6.2916e+01 &  5.5008e+01 \\
  cbc($\boldsymbol{\gamma}^{(2)}$): $e(\boldsymbol{\gamma}^{(3)})$ & 1.4044e+02 & 9.8623e+01 & 6.9702e+01 & 4.9275e+01 &  3.4838e+01 \\
  \hline\hline
  N & 8161 & 16319 & 32633 & 65267 & 130531  \\ \hline
  cbc$2$c: $e(\boldsymbol{\gamma}^{(1)})$  & 9.4743e-05 & 4.9263e-05 & 2.5759e-05 & 1.3568e-05 & 7.2118e-06 \\
  cbc($\boldsymbol{\gamma}^{(1)}$): $e(\boldsymbol{\gamma}^{(1)})$  & 9.4743e-05 & 4.9263e-05 & 2.5759e-05 & 1.3567e-05 & 7.2127e-06 \\
  cbc($\boldsymbol{\gamma}^{(2)}$): $e(\boldsymbol{\gamma}^{(1)})$ & 9.6866e-05 & 5.1703e-05 & 2.6248e-05 & 1.4507e-05 & 7.7439e-06 \\ \hline
  cbc$2$c: $e(\boldsymbol{\gamma}^{(2)})$  & 5.9466e-01 & 4.1992e-01 & 2.9647e-01 & 2.0921e-01 & 1.4753e-01 \\
  cbc($\boldsymbol{\gamma}^{(1)}$): $e(\boldsymbol{\gamma}^{(2)})$ & 3.6458e+00 & 3.2764e+00 & 3.1991e+00 & 3.2104e+00 &  3.3114e+00 \\
  cbc($\boldsymbol{\gamma}^{(2)}$): $e(\boldsymbol{\gamma}^{(2)})$  & 5.9162e-01 & 4.1731e-01 & 2.9421e-01 & 2.0732e-01 & 1.4605e-01 \\ \hline
  cbc$2$c: $e(\boldsymbol{\gamma}^{(3)})$ & 2.4630e+01 & 1.7417e+01 & 1.2317e+01 & 8.7092e+00 & 6.1583e+00 \\
  cbc($\boldsymbol{\gamma}^{(1)}$): $e(\boldsymbol{\gamma}^{(3)})$ & 4.1063e+01 & 3.1297e+01 & 2.7350e+01 & 2.5229e+01 & 2.2864e+01  \\
  cbc($\boldsymbol{\gamma}^{(2)}$): $e(\boldsymbol{\gamma}^{(3)})$ & 2.4629e+01 & 1.7417e+01 & 1.2316e+01 & 8.7088e+00 & 6.1580e+00  \\
  \hline\hline
  N & 261061 & 522127 & 1044257 &  2088511 & 4177051  \\ \hline
  cbc$2$c: $e(\boldsymbol{\gamma}^{(1)})$  & 3.7566e-06 & 1.9798e-06 & 1.0551e-06 & 5.4057e-07 & 2.9352e-07  \\
  cbc($\boldsymbol{\gamma}^{(1)}$): $e(\boldsymbol{\gamma}^{(1)})$  & 3.7567e-06 & 1.9802e-06 & 1.0503e-06 & 5.3995e-07 & 3.0465e-07 \\
  cbc($\boldsymbol{\gamma}^{(2)}$): $e(\boldsymbol{\gamma}^{(1)})$ & 4.3237e-06 & 2.1728e-06 & 1.2567e-06 & 6.4075e-07 & 3.1434e-07  \\ \hline
  cbc$2$c: $e(\boldsymbol{\gamma}^{(2)})$  & 1.0412e-01 & 7.3434e-02 & 5.1821e-02 & 3.6553e-02 &  2.5707e-02 \\
  cbc($\boldsymbol{\gamma}^{(1)}$): $e(\boldsymbol{\gamma}^{(2)})$ & 3.0074e+00 & 2.9431e+00 & 3.1018e+00 & 2.7223e+00 &  2.8919e+00 \\
  cbc($\boldsymbol{\gamma}^{(2)}$): $e(\boldsymbol{\gamma}^{(2)})$  & 1.0284e-01 & 7.2401e-02 & 5.0932e-02 & 3.5823e-02 & 2.5188e-02 \\ \hline
  cbc$2$c: $e(\boldsymbol{\gamma}^{(3)})$  & 4.3546e+00 & 3.0791e+00 & 2.1772e+00 & 1.5395e+00 & 1.0886e+00 \\
  cbc($\boldsymbol{\gamma}^{(1)}$): $e(\boldsymbol{\gamma}^{(3)})$ & 1.8153e+01 & 1.9073e+01 & 1.8373e+01 & 1.4183e+01 & 1.6737e+01  \\
  cbc($\boldsymbol{\gamma}^{(2)}$): $e(\boldsymbol{\gamma}^{(3)})$ & 4.3543e+00 & 3.0788e+00 & 2.1770e+00 & 1.5393e+00 & 1.0884e+00  \\ \hline\hline
\end{tabular}
\caption{This table shows the square worst-case errors using the cbc$2$c construction based on the weights $\boldsymbol{\gamma}^{(1)}$ and $\boldsymbol{\gamma}^{(2)}$, the cbc construction based on the weights $\boldsymbol{\gamma}^{(1)}$ and the cbc construction based on the weights $\boldsymbol{\gamma}^{(2)}$. Here, $e(\boldsymbol{\gamma}^{(w)})$ stands for the worst-case error $e(\mathcal{H}_{\boldsymbol{\gamma}^{(w)}}, P_N(\boldsymbol{g}^\ast))$. We choose $s = 100$, product weights with $\widehat{\gamma}^{(1)}_j = 2^{-j}$, $\widehat{\gamma}_j^{(2)}$ is chosen randomly, $\widehat{\gamma}_j^{(3)} = 1$, $c_1=c_2=2$;}\label{table4}
\end{table}

\end{document}